\documentclass{article}

\usepackage{graphicx}
\usepackage{amscd,amsthm,amsmath,amssymb,enumerate}
\usepackage[mathscr]{eucal}

\newcommand{\nth}{{}^{\text{th}}}
\newcommand{\nl}{\nolimits}

\newcommand{\Z}{\mathbb{Z}}
\newcommand{\Q}{\mathbb{Q}}
\newcommand{\C}{\mathbb{C}}
\newcommand{\F}{\mathbb{F}}
\newcommand{\N}{\mathbb{N}}

\newcommand{\SL}{\text{SL}}

\newcommand{\ol}{\overline}
\newcommand{\ord}{\text{ord}}

\newcommand{\cD}{\mathcal{D}}

\newcommand{\cM}{\mathcal{M}}

\newcommand{\cX}{\mathcal{X}}

\newcommand{\sO}{\mathscr{O}}

\newcommand{\vv}{\mathbf{v}}
\newcommand{\ee}{\mathbf{e}}
\newcommand{\tU}{\widetilde{U}}
\newcommand{\tD}{\widetilde{\mathcal{D}}}
\newcommand{\bA}{\mathbf{A}}

\newcommand{\eps}{\varepsilon}

\newcommand{\dia}[1]{{\langle#1\rangle}}

\theoremstyle{plain}
\newtheorem{theorem}{Theorem}[section]

\newtheorem{proposition}[theorem]{Proposition}
\newtheorem{lemma}[theorem]{Lemma}

\newtheorem{remark}[theorem]{Remark}

\newtheorem*{theorem*}{Theorem}

\begin{document}

\title{Slopes of the $U_7$ Operator Acting on a Space of Overconvergent Modular Forms}
\author{L. J. P. Kilford and Ken McMurdy}
\maketitle

\begin{abstract}
Let $\chi$ be the primitive Dirichlet character of conductor $49$ defined by $\chi(3)=\zeta$, for $\zeta$ a primitive $42^{\text{nd}}$ root of unity. We explicitly \mbox{compute} the slopes of the $U_7$ operator acting on the space of overconvergent \mbox{modular} forms on $X_1(49)$ with weight $k$ and character either $\chi^{7k-6}$ or $\chi^{8-7k}$, \mbox{depending} on the embedding of $\Q(\zeta)$ into $\C_7$. By applying results of \mbox{Coleman,} and of Cohen-Oesterl\'e, we are then able to conclude the slopes of $U_7$ acting on all classical Hecke newforms of the same weight and \mbox{character.}
\end{abstract}

\section{Introduction}\label{Section:Intro}

Let $N$ be an arbitrary positive integer. Suppose that $f$ is a normalized cuspidal Hecke eigenform for $\Gamma_1(7N)$, whose $q$-expansion at $\infty$, $f(q)=\sum_{n=1}^{\infty}a_nq^n$, is defined over a number field $L$. Then $f$ is an eigenform for the $U_7$ operator with eigenvalue $a_7$. We define the {\em slope} of $U_7$ acting on $f$ to be the $7$-adic valuation\footnote{Here we normalize the $7$-adic valuation so that $v(7)=1$.} of $a_7$ viewed as an element of $\C_7$. From this definition it is clear that the slope depends on the embedding of $L$ into $\C_7$.

In particular, suppose now that $L$ contains the cyclotomic field $K=\Q(\zeta)$, where $\zeta$ is a fixed primitive $42^{\text{nd}}$ root of unity. This would necessarily be the case, for example, if $f$ were a newform for $\Gamma_1(49)$ with character $\chi$ defined by $\chi(3)=\zeta$. Over the degree $12$ extension, $K/\Q$, the prime ideal $(7)$ factors as $(7)=(\pi_1)^6(\pi_2)^6$, where
$$\pi_1=-\zeta^8 + \zeta^6 - \zeta^4 + \zeta \quad\text{and}\quad  \pi_2=\zeta^9 + \zeta^8 + \zeta^4 + \zeta^3 - \zeta - 1.$$
Thus there are two types of embeddings of $L$ into $\C_7$, which can be described as follows. Let $K_i=K_{(\pi_i)}$, the completion of $K$ at the prime ideal $(\pi_i)$. The image of $L$ must generate a complete subfield $\hat{L}\subseteq \C_7$ that contains either $K_1$ or $K_2$, and we say that the embedding is of Type 1 or Type 2 accordingly. This is a convenient distinction if we wish to do concrete global calculations over $K$ but draw conclusions over $\C_7$. Alternatively, note that the $42^{\text{nd}}$ cyclotomic polynomial factors over $\F_7$ as
\begin{align*}
\Phi_{42}(x)&=x^{12} + x^{11} - x^9 - x^8 + x^6 - x^4 - x^3 + x + 1\\
                     &=(x+2)^6(x+4)^6.
\end{align*}
Since $v_{(\pi_1)}(\zeta+2)=1$ (in $K$), this implies that the embedding is of Type 1 precisely when $v(\zeta+2)>0$ (in $\C_7$), and of Type 2 when $v(\zeta+4)>0$.

At this point we are able to state the main result of this paper. 

\begin{theorem}\label{Theorem:MainTheorem-Intro}
Let $k$ be an integer greater than $1$. Fix a primitive $42^{\text{nd}}$ root of unity, $\zeta$, and let $\chi$ be the Dirichlet character of conductor $49$ defined by $\chi(3)=\zeta$.

The classical space, $S_k(\Gamma_0(49),\chi^{7k-6})$, is diagonalized by $U_7$ over $K_1$. The slopes of $U_7$ acting on this space are precisely those values in the set,
$$\left\{\tfrac{1}{6}\cdot\left\lfloor\tfrac{9i}{7}\right\rfloor: i \in \N\right\},$$
which are less than $k-1$ (each corresponding to a one-dimensional eigenspace).

The classical space, $S_k(\Gamma_0(49),\chi^{8-7k})$, is diagonalized by $U_7$ over $K_2$. The slopes of $U_7$ acting on this space are precisely those values in the set,
$$\left\{\tfrac{1}{6}\cdot\left\lfloor\tfrac{9i+6}{7}\right\rfloor: i \in \N\right\},$$
which are less than $k-1$ (each corresponding to a one-dimensional eigenspace).
\end{theorem}

Our general approach follows what has become the standard line of attack for slope questions such as these (see \cite[\S 2]{kilford-5slopes} for a survey of related past work). We view the classical forms as a subspace of the overconvergent forms on $X_1(49)$ with the same weight and character. These are defined as sections over a certain rigid-analytic subspace of the modular curve as in \cite{coleman-overconvergent} (see Section \ref{Section:OverconvergentForms}). Using an Eisenstein series, we pull back the overconvergent forms with weight and character to overconvergent forms of weight $0$ on $X_0(49)$, on which a ``twisted'' $U_7$ operator acts with the same eigenvalues (see Section \ref{Section:TwistedU7}). Then, by choosing a ``basis'' for these overconvergent forms (which are really just holomorphic functions on a wide-open disk), the twisted $U_7$ operator can essentially be viewed as an infinite matrix whose characteristic series can be computed explicitly. The bulk of this work is done in Section \ref{Section:ExplicitFormulas}. Finally, the coefficients of the characteristic series give the $U_7$ slopes of all {\em overconvergent} forms with the given weight and character, and then we are able to apply well-known results of Coleman and Cohen-Oesterl\'e to determine which of these forms must have been classical.

There are a couple of important ways, however, in which our work is different than any previous. First of all, analogous explicit slope calculations have only previously been done over genus $0$ modular curves. For example, the work of \cite{kilford-5slopes} is set over $X_0(25)$. Similarly, in \cite{loeffler}, Loeffler focuses primarily on $X_0(p)$ where $p=2$, $3$, $5$, $7$, and $13$. Genus $0$ certainly simplifies the process of describing the matrix representing $U_p$. By working over $X_0(49)$, though, we show that this condition is by no means necessary. A second important distinction in our work is that we do not ultimately restrict our overconvergent forms to an affinoid subdomain in order to apply Serre's theory of compact operators. Instead, we view the wide open disk over which the forms are defined as a residue disk in the stable model for the genus $1$ curve $X_0(49)$. This enables us to ``lift and reduce'' overconvergent forms to meromorphic functions on the good reduction, which makes it possible to argue independence via Riemann-Roch in the proofs of Theorems \ref{Theorem:MainTheorem-Weight1} and \ref{Theorem:MainTheorem-Overconvergent}. Thus, the stable reduction of the modular curve plays a key role in our proof, which may offer a new line of attack for more specific cases or even the general case.

In Section \ref{Section:Verify}, we were able to independently verify our theorem in the weight $2$ case using some very useful data which we found on William Stein's Modular Forms Explorer website. In addition to this acknowledgment, we would also like to express our appreciation for the open source computational software package, SAGE \cite{sage}, which was used for all of our explicit calculations. The files for all of these calculations are available on the second author's website.

\section{Explicit Models}\label{Section:Models}

We will need explicit equations for the modular curves $X_0(7)$ and $X_0(49)$, as well as the moduli-theoretic maps between them and the $j$-line. These can be imported directly from \cite[\S2]{mcmurdy}, but we repeat them here for the convenience of the reader.

For $X_0(7)$, which has genus $0$, we may choose as a parameter the eta quotient $t=(\eta_1/\eta_7)^4$. Like all eta quotients, the divisor of $t$ is supported on the cusps, and in this case given by $(t)=(0)-(\infty)$. Let $\pi_1:X_0(7)\to X(1)$ be the so-called ``forgetful'' map which fixes $q$-expansions at infinity, and let $\pi_7:X_0(7)\to X(1)$ be the map for which $\pi_7^*F(q)=F(q^7)$. Then we have
\begin{align} \label{Eq:Pi1a}
\pi_1^*(j)&=\frac{(t^2+13t+49)(t^2+245t+2401)^3}{t^7}\\ \label{Eq:Pi1b}
                &=1728+\frac{(t^4-10\cdot 7^2t^3-9\cdot 7^4 t^2-2\cdot 7^6t-7^7)^2}{t^7}\\
\pi_7^*(j)&=\frac{(t^2+13t+49)(t^2+5t+1)^3}{t}\\
                &=1728+\frac{(t^4+14t^3+63t^2+70t-7)^2}{t}.
\end{align}
The Atkin-Lehner involution on $X_0(7)$ is also given by $w_7^*t=49/t$. 

From Equation \ref{Eq:Pi1b}, and the fact that $j=1728$ is the only supersingular $j$-invariant (mod $7$), we see that the unique supersingular annulus is the region where $0<v_7(t)<2$. From Equation \ref{Eq:Pi1a}, we see that $X_0(7)$ has two elliptic points of order $3$, defined by $t^2+13t+49$. From the Newton polygon of this quadratic, we see that the $t$-coordinates of the two elliptic points have $7$-adic valuation $0$ and $2$. Thus they lie in the ordinary locus, with one on either ``side'' of the supersingular annulus (see \cite[Fig. 1]{mcmurdy} for a picture). For consistency, we will always denote these elliptic points as $e_1$ and $e_2$, where $v(t(e_1))=0$ and $v(t(e_2))=2$. This is an important point for us, particularly since the elliptic points occur in the support of the Eisenstein series which we use to pass between overconvergent forms of different weight and character (see Proposition \ref{Prop:Psi-Iso}). 

For the genus $1$ modular curve, $X_0(49)$, we may choose as parameters the two eta quotients, $x=\eta_1/\eta_{49}$ and $y=(\eta_7/\eta_{49})^4$. These are also supported on the cusps and have the following divisors:
$$(x)=2(0)-2(\infty),\qquad (y)=(0)+{\sum_{i=1}^6(C_{7,i})}-7(\infty).$$
Here, as in \cite[\S2]{mcmurdy}, we use $C_{7,i}$ to represent those cusps whose underlying generalized elliptic curve is the N\'{e}ron $7$-gon.
The equation for $X_0(49)$ in terms of these parameters is given by
\begin{equation}\label{Eq:X49model}
y^2-7xy(x^2+5x+7)-x(x^6+7x^5+21x^4+49x^3+147x^2+343x+343)=0.
\end{equation}
Defining $\pi_1,\pi_7:X_0(49)\to X_0(7)$ as above, we clearly have $\pi_7^*t=y$. From \cite[\S2]{mcmurdy} we also have $\pi_1^*t=x^4/y$ and $w_{49}^*(x,y)=(7/x,49y/x^4)$.
This curve also has two elliptic points, $\hat{e}_1$ and $\hat{e}_2$, which lie over $e_1$ and $e_2$ via either map. The fibers over $e_1$ and $e_2$ figure prominently in our work, and thus are described in great detail in Lemma \ref{Lemma:EllipticPoints}.

At times, it will be useful to have a Weierstrass equation for $X_0(49)$, and in this case we take
$$z=\frac{y-\tfrac{7}{2}x(x^2+5x+7)}{x^2+7x+7}.$$
This results in the equation
\begin{equation}\label{Eq:WeierstrassModel}
z^2=x(x^2+\tfrac{21}{4}x+7).
\end{equation}
Moreover, a good reduction model $\cX$ for $X_0(49)$ exists over any Galois extension of $\Q_7$ containing a root $\alpha$ of $x^4+7$. In particular, if we let $z=\alpha^3Z$ and $x=\alpha^2 X$, we obtain the equation
\begin{equation}\label{Eq:X49Reduction}
Z^2=X(X^2-1)\pmod{\alpha^2}.
\end{equation}

\section{Eisenstein Series}\label{Section:Eisenstein}

In order to translate forms with character to forms on $X_0(49)$, we will use various Eisenstein series on $X_1(49)$. In this section, we define these Eisenstein series using the well-known $q$-expansion formula (see \cite[\S2.2]{diamond-im}, for example), and compute their divisors using Shimura's theory of divisors \cite[\S2.4]{shimura}. This enables us to avoid holomorphicity issues when dividing by these forms. In all cases, we use $B_{k,\eps}$ to represent the generalized Bernoulli number for weight $k$ and character $\eps$ (as defined in \cite[\S 2.2]{diamond-im}).

\begin{proposition}\label{Prop:E1tau}
Let $\tau$ be an odd character of conductor $7$, defined by $\tau(3)=\beta$ for $\beta$ some primitive $6\nth$ root of unity. Let $E_{1,\tau}$ be the weight $1$ Eisenstein series on $X_1(7)$ defined by
$$E_{1,\tau}(q)=1-\frac{2}{B_{1,\tau}}\sum_{n=1}^\infty{\Big(\sum_{d|n}\tau(d)\Big)}q^n.$$
The divisor of $E_{1,\tau}^6$, considered as a modular form on $X_0(7)$, is $4(e_{\beta})$, where $e_{\beta}\in X_0(7)$ is the elliptic point with $t(e_{\beta})=3\beta-8$.
\end{proposition}

\begin{proof}
Let $F$ be the weight $2$ meromorphic form on $X_0(7)$ which corresponds to the differential $-dt$ by the well-known correspondence between weight $2$ forms and differentials. Then since $(dt)=-2(\infty)$, we may apply \cite[Prop. 2.16]{shimura} to see that the divisor of $F$ as a modular form is given by
$$(F)=(0)-(\infty)+\tfrac{2}{3}(e_1)+\tfrac{2}{3}(e_2).$$
Therefore, since the Eisenstein series is holomorphic, $g:=E_{1,\tau}^6/F^3$ must be a function on $X_0(7)$ whose divisor satisfies $(g)\geq 3(\infty)-3(0)-2(e_1)-2(e_2)$. Comparing $q$-expansions of functions in $L(4(\infty))$, which is finite dimensional and spanned by $\{1,t,t^2,t^3,t^4\}$, we find that
$$t^3(t^2+13t+49)^2\cdot g=(t-(3\beta-8))^4.$$
\end{proof}

\begin{lemma}\label{Lemma:E7tau}
Let $\tau$ be as above. Let $E_{7,\tau}$ be the weight $7$ Eisenstein series on $X_1(7)$ defined by
$$E_{7,\tau}(q)=1-\frac{14}{B_{7,\tau}}\sum_{n=1}^\infty{\Big(\sum_{d|n}\tau(d)d^6\Big)}q^n.$$
The divisor of $E_{7,\tau}^6$, considered as a modular form on $X_0(7)$, is given by $(E_{7,\tau}^6)=4(e_{\beta})+6(0)+6(P_1)+6(P_2)+6(P_3)$, where $e_{\beta}$ is as above and the $t$-coordinates of the $P_i$ satisfy
\begin{multline*}
P(t)=16346149t^3 + (32722347\beta + 179781490)t^2\\
+ (178382295\beta + 587942474)t + (141531747\beta + 388829945)=0.
\end{multline*}
\end{lemma}
\begin{proof}
Take $F$ as above, and compare $q$-expansions to see that the following two functions in $L(28(\infty))$ are equal.
$$\frac{E_{7,\tau}^6t^{21}(t^2+13t+49)^{14}}{F^{21}}=(t-(3\beta-8))^4t^6\left(\frac{P(t)}{16346149}\right)^6$$
\end{proof}

\begin{proposition}\label{Prop:E1chi}
Let $\chi$ be an odd, primitive Dirichlet character of conductor $49$, defined by $\chi(3)=\zeta$ where $\zeta$ is a primitive $42^{\text{nd}}$ root of unity. Let $E_{1,\chi}$ be the weight $1$ Eisenstein series on $X_1(49)$ defined by
$$E_{1,\chi}=1-\frac{2}{B_{1,\chi}}\sum_{n=1}^{\infty}{\Big(\sum_{d|n}\chi(d)\Big)}q^n.$$
Let $\hat{e}_{\zeta}$ be the elliptic point of $X_0(49)$ with $x(\hat{e}_{\zeta})=3\zeta^7-1$.
The divisor of $E_{1,\chi}^{42}$, as a modular form on $X_0(49)$, is given by $(E_{1,\chi}^{42})=28(\hat{e}_{\zeta})+42(Q)+6\sum_{i=1}^6{i(C_{7,i})}$ (with correct ordering of the $C_{7,i}$), where 
\begin{multline*}
1849\cdot x(Q)=-2040\zeta^{11} - 2342\zeta^{10} + 266\zeta^9 + 3903\zeta^8 + 883\zeta^7\\
- 2873\zeta^6 - 3359\zeta^5 + 2840\zeta^4 + 2968\zeta^3 + 1515\zeta^2 - 3229\zeta - 5616.
\end{multline*}
\end{proposition}


\begin{proof}
Take $\beta=\zeta^7$, so that $\beta$ is a primitive $6\nth$ root of unity and we have $\chi^7=\tau$ (with $\tau$ as in Lemma \ref{Lemma:E7tau}). So by Lemma \ref{Lemma:E7tau}, $g:=E_{1,\chi}^7/E_{7,\tau}$ can be viewed as a function on $X_0(49)$ whose divisor satisfies 
$$6(g)\geq\pi_1^*(-4(e_{\beta})-6(0)-6(P_1)-6(P_2)-6(P_3))$$
(where $e_{\beta}$ is the elliptic point on $X_0(7)$ with $t(e_{\beta})=3\beta-8$). Taking into account that $\pi_1^*(0)=7(0)$ while $\pi_1^*(\infty)=(\infty)+\sum_{i=1}^6(C_{7,i})$, we see that
$$h:=g\cdot (t-(3\beta-8))P(t) y^4x^2\in L(36(\infty)).$$
Therefore, as this space is finite dimensional and spanned by $\{1,x,z,x^2,xz,\dots,x^{18}\}$, we may compare $q$-expansions to write $h=f_1(x)+zf_2(x)$ for polynomials $f_i(x)$ over the cyclotomic field.

The divisor of $E_{1,\chi}$ will now follow if we can compute the divisor of $h$. So we first substitute $z=-f_1(x)/f_2(x)$ into Equation \ref{Eq:WeierstrassModel} to determine the $x$-coordinates of the zeroes of $h$, and then plug back in to get $z$ (and subsequently $y$). Thus we see that
$$(h)=(0)+{\sum_{i=1}^6i(C_{7,i})} + 7(Q) + 5(\hat{e}_{\zeta})+(a_{\beta})+(b_{\beta})-36(\infty),$$
where $\pi_1^*(e_{\beta})=(\hat{e}_{\zeta})+3(a_{\beta})+3(b_{\beta})$ (see Lemma \ref{Lemma:EllipticPoints} for more explanation). 
In conclusion, we have
\begin{align*}
(E_{1,\chi}^{42})&=6(g)+\pi_1^*(E_{7,\tau}^6)\\
			&=6(h)-6\pi_1^*(t-(3\beta-8))-6\pi_1^*(P(t))-24(y)-12(x)+\pi_1^*(E_{7,\tau}^6)\\
			&=28(\hat{e}_{\zeta})+42(Q)+6\sum_{i=1}^6{i(C_{7,i})}.
\end{align*}
\end{proof}

\section{Overconvergent Modular Forms}\label{Section:OverconvergentForms}

In order to draw conclusions about slopes of classical modular forms, it is imperative that we be able to apply the main theorem from \cite{coleman-overconvergent} which can be rephrased as follows.

\begin{theorem}[\cite{coleman-overconvergent},Theorem 1.1]
Every $p$-adic overconvergent form of weight $k$ and level $p^n$ with slope strictly less than $k-1$ is classical.
\end{theorem}
\noindent 
So we must be careful to define our space of overconvergent modular forms on $X_1(49)$ in a way which is consistent with the intrinsic definition given in \cite[\S1]{coleman-overconvergent}. Adapting this definition to our situation, we first let $f_2:E_1(49)\to X_1(49)$ be the universal generalized elliptic curve\footnote{The existence of the universal curve over $X_1(M)$ when $M>4$ follows easily from \cite[IV.3]{deligne-rapoport}. See \cite[Proposition 2.1]{gross-tameness}, for example.} over $X_1(49)$ and let $\omega={f_2}_*\Omega^1_{E_1(49)/X_1(49)}$. Then for $k\in\Z$, we define the space of (holomorphic) overconvergent modular forms of weight $k$ on $X_1(49)$ by
$$M_k(49):=\omega^k(W_1(49)),$$
where $W_1(49)$ is a certain wide open subspace of the curve. In order to do our calculations on $X_0(49)$, we must determine the image of this $W_1(49)$ under the forgetful map from $X_1(49)$ to $X_0(49)$. 

According to \cite[\S1]{coleman-overconvergent}, $W_1(p^2)$ lies over $W_1(p)$, which in turn is the connected component containing the cusp, $\infty$, of the rigid subspace of $X_1(p)$ where $v(E_{p-1})<p/(p+1)$. Here $E_k$, for $k\geq 4$ even, is the well-known lifting of the Hasse invariant to a weight $k$ Eisenstein series for $\SL_2(\Z)$, as described in \cite[\S2.1]{katz}. Recall from \cite[\S3]{katz} (see also \cite[\S3]{buzzard-analytic}), that for a given elliptic curve this condition on $E_{p-1}$ is equivalent to the existence of the canonical subgroup. Thus, $W_1(p)$ is simply the rigid subspace of $X_1(p)$ whose points correspond to pairs $(E,Q)$, where $E$ is an elliptic curve and $Q$ is a point which generates the canonical subgroup of $E$. Alternatively, in the language of \cite[\S4]{buzzard-analytic}, $W_1(p)$ is the wide open neighborhood of the cusp, $\infty$, which extends into each supersingular annulus precisely as far as the too-supersingular circle. By valuation considerations, as in the proof of \cite[Claim 2.2]{mcmurdy}, it is clear then that the forgetful image of $W_1(7)$ in $X_0(7)$ is simply the disk $\cD$ described on our explicit model by $v(t)<7/4$ (this is the maximal open disk upon which $\pi_1$ has degree $1$). Now, to move up to $W_1(p^2)$, we are to take the inverse image of $W_1(p)$, under the map $\Phi:X_1(p^2)\to X_1(p)$ which is given in moduli-theoretic terms by $\Phi(E,Q)=(E/(pQ),\bar{Q})$. Therefore, the forgetful image of $W_1(49)$ in $X_0(49)$ is precisely $\pi_7^{-1}(\cD)$.\footnote{In the language of \cite[\S 3B]{coleman-mcmurdy}, the forgetful image of $W_1(p^2)$ in $X_0(p^2)$ is $W_{2\,0}$. See also Theorem 5.3 of \cite{coleman-mcmurdy}.} Given that $y=\pi_7^*t$, this is just the wide open disk $\tilde{\cD}\subseteq X_0(49)$ described by $v(y)<7/4$, or equivalently by $v(x)<1/2$ (this region is shown to be a disk in the proof of \cite[Claim 2.4 (i)]{mcmurdy}). 

\begin{lemma}\label{Lemma:EllipticPoints}
Let $e_1$ and $e_2$ be the two elliptic points on $X_0(7)$ as described in Section \ref{Section:Models}. The $\pi_1$ and $\pi_7$ fibers over these points satisfy the following conditions:
\begin{enumerate}[(i)]
\item  $\pi_1^{-1}(e_1)\cap\tD=\pi_7^{-1}(e_1)\cap\tD=\{\hat{e}_1\}$
\item $\pi_1^{-1}(e_2)\cap\tD=\pi_7^{-1}(e_2)\cap\tD=\emptyset$.
\end{enumerate}
\end{lemma}
\begin{proof}
It is straightforward to verify this lemma by completely explicit means. In particular, let $\gamma$ be a root of $t^2+13t+49$. Then $\gamma$ is the $t$-coordinate of either $e_1$ or $e_2$, depending on whether $v(\gamma)=0$ or $2$. 

Since $\pi_1:X_0(49)\to X_0(7)$ is determined by $\pi_1^*t=x^4/y$, we can compute $\pi_1^{-1}(e_i)$ by substituting $\gamma^{-1}x^4$ for $y$ in Equation \ref{Eq:X49model}. The resulting polynomial in $x$ is a constant multiple of
$$x(x - (\gamma + 7))(x^2 + ((7/3)\gamma + 56/3)x + (7\gamma + 49))^3.$$
Setting $v(\gamma)=2$, we see that $\pi_1^*(e_2)=(\hat{e}_2)+3(a_2)+3(b_2)$, where $v(x(\hat{e}_2))=1$ and $v(x(a_2))=v(x(b_2))=1/2$ (from the Newton polygon of the quadratic). So none of these points lie on $\tD$. On the other hand, if we set $v(\gamma)=0$, we find that $\pi_1^*(e_1)=(\hat{e}_1)+3(a_1)+3(b_1)$, where $v(x(\hat{e}_1))=0$ while $v(x(a_1))=v(x(b_1))=1/2$. So $\hat{e}_1\in\tD$, but the other two (non-elliptic) points in $\pi_1^{-1}(e_1)$ are not.

Since $\pi_7:X_0(49)\to X_0(7)$ is determined by $y=\pi_7^*t$, we may compute $\pi_7^*(e_i)$ by substituting $y=\gamma$ into Equation \ref{Eq:X49model}. This results in the polynomial,
$$(x-(\gamma+7))(x^2 + ((1/3)\gamma + 14/3)x + \gamma + 7)^3,$$
and the rest of the reasoning is the same. Note that $\gamma^{-1}(\gamma+7)^4=\gamma$. So it really is the same elliptic point, $\hat{e}_i$, which lies over $e_i$ via both $\pi_1$ {\em and} $\pi_7$.
\end{proof}

It is worth pointing out here that $\tilde{\cD}$ is also a residue class in our good-reduction model for $X_0(49)$. This is a fact which we exploit in our proof of the main theorem, and it is not at all a coincidence. Indeed, it is a consequence of \cite[Theorem 5.3]{coleman-mcmurdy}. More generally, the forgetful image of $W_1(p^2)$ in $X_0(p^2)$ is always the unique wide open neighborhood of $\infty$ which extends into the supersingular locus precisely far enough to contain one full residue class of each supersingular component in the stable model. So just as the arithmetic of our good reduction model for $X_0(49)$ is used in our proof, it is reasonable to expect that the stable reduction of $X_0(p^2)$ might be a key component in a more general proof.

\subsection{Twisted $U_7$ Operator}\label{Section:TwistedU7}

From \cite[\S1]{coleman-overconvergent}, the Hecke operator $U_p$ can be extended to a linear operator on $M_k(p^n)$ which acts on $q$-expansions at infinity in the usual way, taking $\sum_n a_n q^n$ to $\sum_n a_{np} q^n$. The diamond-bracket operators, $\dia{d}$ for $d\in(\Z/p^n\Z)^*$, also extend naturally and can be used to define character subspaces of $M_k(p^n)$ which are preserved by $U_p$. In particular, let $k$ be an integer and let $\eps$ be a Dirichlet character mod $49$. Then we define $M_{k,\eps}(49)\subseteq M_k(49)$ to be the subspace defined by $F|\dia{d}=\eps(d)F$. We want to compute the spectrum of the linear operator $U_7$ on $M_{k,\chi\tau^{k-1}}(49)$, where $\chi$ and $\tau$ are as in the previous section. The following proposition shows that this space can be identified with the space of rigid-analytic functions on the disk $\tilde{\cD}\subseteq X_0(49)$, the space of functions which we denote from this point on by $\cM_0$.

\begin{proposition}\label{Prop:Psi-Iso}
Let $\chi$ be an odd primitive Dirichlet character of conductor $49$, and $\tau$ an odd character of conductor $7$, determined by $\chi(3)=\zeta$ and $\tau(3)=\beta$ as in Section \ref{Section:Eisenstein}. There is an isomorphism, $\Psi:\cM_0\to M_{k,\chi\tau^{k-1}}(49)$, given by 
$$\Psi(F)=F\cdot E_{1,\chi}\cdot E_{1,\tau}^{k-1}\cdot (t^{-1}-t(e_1)^{-1})^{-d_k(\chi,\tau)}$$
for some $d_k(\chi,\tau)\in\Z$.
\end{proposition}

\begin{proof}
The character of $\Psi(F)$ is clearly correct. So if both Eisenstein series were holomorphic and non-vanishing on $W_1(49)$, we could simply take $d_k(\chi,\tau)=0$ and the statement would follow. This is nearly the case, as we will show that the only zeroes of $E_{1,\chi}\cdot E_{1,\tau}^{k-1}$ on $W_1(49)$, if any, are those lying over $\hat{e}_1$. So then we may exploit the fact that
\begin{equation}\label{Eq:HolomorphyFactor}
\pi_1^*(t^{-1}-t(e_1)^{-1})=(\hat{e}_1)+3(a_1)+3(b_1)-7(0),
\end{equation}
 and choose $d_k(\chi,\tau)$ so as to cancel out these zeroes without introducing any new zeroes or poles on $W_1(49)$.


So we begin by considering the zeroes of $E_{1,\chi}$ and $E_{1,\tau}$ which do not lie over either elliptic point of $X_0(7)$. In particular, from Proposition \ref{Prop:E1chi}, we must consider the special point $Q$ and the six cusps denoted by $C_{7,i}$. The cusps can be dealt with easily, since $y$ vanishes at these points and $v(y)<7/4$ on $\tilde{\cD}$. To eliminate $Q$, we consider the roots of the minimal polynomial for $x(Q)$:
\begin{multline*}
1849x^{12} + 35336x^{11} + 293356x^{10} + 1345736x^9 + 3511340x^8\\
 + 4649708x^7 + 4436705x^6 + 32547956x^5 + 172055660x^4\\
  + 461587448x^3 + 704347756x^2 + 593892152x + 217533001.
\end{multline*}
The Newton polygon of this polynomial is the straight line from $(0,0)$ to $(12,-6)$. So all of its roots have valuation $1/2$ regardless of the embedding into $\C_7$ (the completion of a fixed algebraic closure of $\Q_7$). Therefore, $v(x(Q))$ must equal $1/2$, and $Q\notin \tilde{\cD}$.

Now we consider the zeroes of $E_{1,\chi}$ and $E_{1,\tau}$ which do lie over some elliptic point. By Lemma \ref{Lemma:EllipticPoints} and Propositions \ref{Prop:E1tau} and \ref{Prop:E1chi}, any such zeroes will lie on $W_1(49)$ if only if they lie over $\hat{e}_1$, and either $\hat{e}_{\beta}=\hat{e}_1$ or $\hat{e}_{\zeta}=\hat{e}_1$ (or both). Let $\delta_{\beta}=1$ if $\hat{e}_{\beta}=\hat{e}_1$, $\delta_{\beta}=0$ otherwise, and similarly for $\delta_{\zeta}$. Then $\ord_{\hat{e}_1}(E_{1,\chi}\cdot E_{1,\tau}^{k-1})$ makes sense, and is given by
$$\ord_{\hat{e}_1}(E_{1,\chi}\cdot E_{1,\tau}^{k-1})=\tfrac{2}{3}\delta_{\zeta}+\tfrac{2}{3}(k-1)\delta_{\beta}.$$
Although this may not be an integer, we can set $d_k(\chi,\tau)=\lfloor{\ord_{\hat{e}_1}(E_{1,\chi}\cdot E_{1,\tau}^{k-1})}\rfloor$. 

Putting all of the preceding information together, we are now in a position to argue that $\Psi$ is an isomorphism. Let $G_k(\chi,\tau)$ be the factor by which we multiply $F$ to get $\Psi(F)$. The fact that $\Psi$ is at least an injection follows immediately from the fact that $G_k(\chi,\tau)$ is a meromorphic form on $X_1(49)$ with poles only over $a_1$ and $b_1$, and these points do not lie on $\tD$ by Lemma \ref{Lemma:EllipticPoints}. Moreover, taking $F$ to $F/G_k(\chi,\tau)$ defines an inverse function from $M_{k,\chi\tau^{k-1}}(49)$ to $\cM_0$. Indeed, the only possible zeroes of $G_k(\chi,\tau)$ which lie on $W_1(49)$ are the points over $\hat{e}_1$, and we have chosen $d_k(\chi,\tau)$ so that $0\leq \ord_{\hat{e}_1}G_k(\chi,\tau)<1$. So $F/G_k(\chi,\tau)$ is a meromorphic function on $\tD$, holomorphic away from $\hat{e}_1$ and with $\ord_{\hat{e}_1}(F/G_k(\chi,\tau))>-1$. Hence this is a holomorphic function in $\cM_0$, and we have shown that $\Psi$ is an isomorphism.
\end{proof}

Let $V$ be the map from $M_1(p^n)$ to $M_1(p^{n+1})$ for which $V(F(q))=F(q^p)$. As is explained in \cite[\S B3]{coleman} (see also \cite[(3.3)]{coleman-old}), $U_p$ and $V$ interact according to the formula, $U_p(F\cdot V(G))=G\cdot U_p(F)$. So if we pull back $U_7$ via $\Psi$ to a linear operator on $\cM_0$, we arrive at the operator $\Psi^{-1}\circ U_7\circ \Psi$ given by
\begin{align*}
\Psi^{-1}\circ U_7\circ \Psi(F)&=\frac{U_7\left(F\cdot E_{1,\chi}\cdot E_{1,\tau}^{k-1}\cdot (t^{-1}-t(e_1)^{-1})^{-d_k(\chi,\tau)}\right)}{E_{1,\chi}\cdot E_{1,\tau}^{k-1}\cdot (t^{-1}-t(e_1)^{-1})^{-d_k(\chi,\tau)}}\\
                                                   &=E_{1,\chi}^{-1}\cdot U_7\left(F\cdot E_{1,\chi}\cdot\frac{E_{1,\tau}^{k-1}}{V(E_{1,\tau})^{k-1}}\cdot\frac{(y^{-1}-t(e_1)^{-1})^{d_k(\chi,\tau)}}{(t^{-1}-t(e_1)^{-1})^{d_k(\chi,\tau)}}\right).
\end{align*}
Instead of applying this operator directly to compute the spectrum of $U_7$ on $M_{k,\chi\tau^{k-1}}(49)$, we choose for convenience to work with the following ``twisted'' $U_7$ operator on $\cM_0$.
\begin{equation}\label{Def-TwistedU7}
\tU_7(F)=E_{1,\chi}^{-1}\cdot U_7(F\cdot E_{1,\chi})\cdot\left(\frac{E_{1,\tau}}{V(E_{1,\tau})}\right)^{k-1}\cdot\left(\frac{y^{-1}-t(e_1)^{-1}}{t^{-1}-t(e_1)^{-1}}\right)^{d_k(\chi,\tau)}
\end{equation}
Separating out the $\tau$ part simplifies our argument greatly, and the following proposition shows that $U_7$ and $\tU_7$ have precisely the same eigenvalues.

\begin{proposition}\label{Prop:TwistedEigenvalues}
The linear operators $\Psi^{-1}\circ U_7\circ \Psi$ and $\tilde{U}_7$ on $\cM_0$ have precisely the same eigenvalues, and isomorphic eigenspaces  for each eigenvalue.
\end{proposition}

\begin{proof}
Suppose $F\in\cM_0$ is an eigenform for $\Psi^{-1}\circ U_7\circ \Psi$ with eigenvalue $\lambda\in\C_7$.
We claim that 
$$G:=F\cdot\left(\frac{E_{1,\tau}}{V(E_{1,\tau})}\right)^{k-1}\cdot\left(\frac{y^{-1}-t(e_1)^{-1}}{t^{-1}-t(e_1)^{-1}}\right)^{d_k(\chi,\tau)}$$
is also an eigenform for $\tU_7$ with eigenvalue $\lambda$.

First we must show that $G$ is in fact a form in $\cM_0$. Recall that $E_{1,\tau}$ is a form on $X_1(7)$. So although $V$ raises the level, $V(E_{1,\tau})$ is still a form on $X_1(49)$. Then, since $V$ preserves characters and weight, the quotient $E_{1,\tau}/V(E_{1,\tau})$ is a (meromorphic) weight $0$ form on $X_1(49)$ with trivial character and therefore can be viewed as a function on $X_0(49)$. The only remaining question is whether $G$ has any poles on $\tD$ which were introduced when we divided by $V(E_{1,\tau})$ and $(t^{-1}-t(e_1)^{-1})$. It does not, and this follows directly from Lemma \ref{Lemma:EllipticPoints}. In particular, $\hat{e}_1$ is in both $\pi_1^{-1}(e_1)$ and $\pi_7^{-1}(e_1)$ (and unramified for both). So it must occur as a zero of the denominator precisely as many times as it does for the numerator. Thus, $G$ is holomorphic on $\tD$.

Now we compute $\tU_7(G)$ to show that $G$ is an eigenvector with eigenvalue $\lambda$.
\begin{align*}
\tU_7(G)&=E_{1,\chi}^{-1}\cdot U_7(G\cdot E_{1,\chi})\cdot\left(\frac{E_{1,\tau}}{V(E_{1,\tau})}\right)^{k-1}\cdot\left(\frac{y^{-1}-t(e_1)^{-1}}{t^{-1}-t(e_1)^{-1}}\right)^{d_k(\chi,\tau)}\\
              &=\left(\Psi^{-1}\circ U_7\circ\Psi(F)\right)\cdot\left(\frac{E_{1,\tau}}{V(E_{1,\tau})}\right)^{k-1}\cdot\left(\frac{y^{-1}-t(e_1)^{-1}}{t^{-1}-t(e_1)^{-1}}\right)^{d_k(\chi,\tau)}\\
              &=\lambda F\cdot\left(\frac{E_{1,\tau}}{V(E_{1,\tau})}\right)^{k-1}\cdot\left(\frac{y^{-1}-t(e_1)^{-1}}{t^{-1}-t(e_1)^{-1}}\right)^{d_k(\chi,\tau)}=\lambda G
\end{align*}
So at this point we have constructed an injection from the $\lambda$-eigenspace of $\Psi^{-1}\circ U_7\circ\Psi$ into the $\lambda$-eigenspace of $\tU_7$. The argument is identical for the other direction.
\end{proof}

\section{Explicit Formulas for $\tU_7$ in the Weight $1$ Case}\label{Section:ExplicitFormulas}

Recalling the notation of Section \ref{Section:Intro}, let $L$ be a number field which contains the cyclotomic field $K=\Q(\zeta_{42})$. Let $\hat{L}$ be the finite extension of $\Q_7$ which is generated by the embedding of $L$ into $\C_7$. So $\hat{L}$ must contain either $K_1$ or $K_2$, and we say that the embedding is of Type $1$ or Type $2$ accordingly. Now suppose that $L$ also contains a root $\alpha$ of $x^4+7$. Then the parameter, $s=\alpha/t$, identifies $\tD$ with the wide open unit disk $B_{\hat{L}}(1)$. In other words, the ring of analytic functions on $\tD$ over $\hat{L}$ is given by
$$A_{\hat{L}}(\tD)=\left\{\sum_{n=0}^{\infty}a_n s^n: \text{ $a_n\in \hat{L}$, $\lim_{n\to\infty} |a_n|r^n=0$ if $0\leq r<1$}\right\}.$$
Our overall strategy is essentially to represent the linear operator $\tU_7$ on $A_{\hat{L}}(\tD)$ as an infinite matrix by writing it in the ``basis'' $\{s,s^2,s^3,\dots\}$. 
Therefore, the ultimate goal of this section is to arrive at an explicit formula for $\tU_7(s^i)$.
Initially, we assume for convenience that $k=1$, so that $d_k(\chi,\tau)=0$ and $\tU_7$ simplifies to
$$\tU_7(F)=E_{1,\chi}^{-1}\cdot U_7(F\cdot E_{1,\chi}).$$
As a result of Proposition \ref{Prop:TwistedEigenvalues}, the contribution of $E_{1,\tau}$ will be easy to take into account later.

To be clear, $A_{\hat{L}}(\tD)$ is {\em not} a $p$-adic Banach space, and $\{s^i\}$ is {\em not} a true Banach basis. However, the structure is still quite nice in other ways which can be exploited. In particular, the sup norm, which we denote by $|\cdot|_{sup}$ (or $|\cdot|$ when the context is clear), can be defined on the (Banach) subspace of $A_{\hat{L}}(\tD)$ consisting of those functions with bounded valuation, by
$$|f|_{sup}=\max_{x\in\tD(\C_7)}|f(x)|=\max_{n}|a_n|.$$
If we set
\begin{align*}
A_{\hat{L}}^{o}(\tD)&=\{f\in A_{\hat{L}}(\tD):|f|_{sup}\leq 1\}\\
A_{\hat{L}}^+(\tD)&=\{f\in A_{\hat{L}}(\tD):|f|_{sup}<1\}\\
\ol{A_{\hat{L}}(\tD)}&=A_{\hat{L}}^{o}(\tD)/A_{\hat{L}}^+(\tD),
\end{align*}
then $\ol{A_{\hat{L}}(\tD)}\cong\F_{\hat{L}}[[s]]$, where $\F_{\hat{L}}$ is the residue field of $\hat{L}$. Moreover, if we take $\cX$ to be the good reduction model of $X_0(49)$ introduced in Section \ref{Section:Models}, and let $P$ be the smooth point at infinity on $\ol{\cX}$ which is the reduction of $\tD$, there is a natural isomorphism between $\ol{A_{\hat{L}}(\tD)}$ and $\hat{\sO}_{\ol{\cX},P}$ (see \cite[Prop. 2.8]{coleman-mcmurdy}, for example). This connection between analytic functions on the disk and functions in the stalk of a smooth point on the stable reduction is a key tool in our proof of the main theorem for overconvergent forms. Thus we highlight it with the following formal remark.

\begin{remark}
As Coleman shows in \cite{coleman-overconvergent}, overconvergent forms naturally live on the wide open $W_1(p^n)$. However, this space is usually restricted down to an affinoid so that spectral theory on Banach spaces may be applied. Our approach is quite different. In some sense, we lift the overconvergent forms up to an affinoid which contains $W_1(p^n)$ as a residue class. Thus we are able to take advantage of arithmetic on the reduction of this affinoid.
\end{remark}

\subsection{Calculation of $\tU_7(x^j)$ for $j=1,\dots,6$}

Although our ultimate goal is to find an explicit formula for $\tU_7(s^i)$ for each $i=1,2,\dots$, it is difficult to do this directly because the divisor of $s$ on $X_0(49)$ is $(\infty)+\sum_{m=1}^6{(C_{7,m})}-7(0)$. In Appendix \ref{Section:Up-Meromorphic}, we show how $U_p$ affects poles at the cusps, and it follows that $\tU_7(s^i)$ will necessarily have a pole of order $49i$ at the cusp $0$. So this approach becomes computationally problematic for even small $i$. As it turns out, it is much easier to compute $\tU_7(x^j)$ for $j=1,\dots,6$ first, and then derive formulas for $\tU_7(s^i)$ by using the following reasoning.

Recall that $U_p(F\cdot V(G))=G\cdot U_p(F)$. Applying this to our situation, since $g(y)=V(g(t))$ for any rational function $g$, we have
\begin{align}
	  \tU_7(g(y)F)	&=E_{1,\chi}^{-1}\cdot U_7(g(y)\cdot F\cdot E_{1,\chi})\notag\\
				&=E_{1,\chi}^{-1}\cdot g(t)\cdot U_7(F\cdot E_{1,\chi})=g(t)\tU_7(F).\label{Eq:ColemanTwistedTrick}
\end{align}
But the function field of $X_0(49)$, even over the global field $K$, is a degree $7$ extension of $K(y)$, and can be viewed as a vector space with basis $\{x^6,x^5,\dots,x,1\}$. So {\em any} weight $0$ form, i.e. function on $X_0(49)$, can be written as
$$F=g_6(y)x^6+g_5(y)x^5+\cdots+g_1(y)x+g_0(y),$$
where the $g_j$ are rational functions. Then by linearity and Equation \ref{Eq:ColemanTwistedTrick} we have
$$\tU_7(F)=g_6(t)\tU_7(x^6)+g_5(t)\tU_7(x^5)+\cdots+g_1(t)\tU_7(x)+g_0(t).$$
Thus, if we obtain explicit formulas for $\tU_7(x^j)$ for $j=1,\dots,6$ first, then in some sense we get for free a completely explicit formula for $\tU_7$.

\begin{proposition}\label{Prop:GlobalUxi}
Let $Q$ be the zero of $E_{1,\chi}$ given in Proposition \ref{Prop:E1chi}, and let $x_Q$ be its $x$-coordinate. Then
\begin{align*}
\tfrac{y(x-x_Q)}{x}\cdot \tU_7(x^j)\in L(7(\infty))&=\text{Span}\{1,x,z,x^2,xz,x^3,x^2z\},\quad j=1,2,3\\
\tfrac{y^2(x-x_Q)}{x}\cdot \tU_7(x^j)\in L(15(\infty))&=\text{Span}\{1,x,z,x^2,\dots,x^7,x^6z\},\quad j=4,5,6.
\end{align*}
\end{proposition}

\begin{proof}
We have assumed that $k=0$, and hence $\tU_7(F)=E_{1,\chi}^{-1}\cdot U_7(F\cdot E_{1,\chi})$. From this expression it is clear that $\tU_7(F)$ can only have poles at the zeroes of $E_{1,\chi}$ and at points which arise from the poles of $F$ through a $p$-isogeny of the corresponding elliptic curve as in Appendix \ref{Section:Up-Meromorphic}. But in this case $(x^j)=2j(0)-2j(\infty)$. So we only need to consider the orders at the cusps, at $\hat{e}_1$, and at $Q$. Applying a slight variant of Lemma \ref{Lemma:Up-Meromorphic}, we see that $y$ and $y^2$ suffice to move all cuspidal poles to $\infty$ (being able to divide by $x$ is a ``coincidence'' which one sees after comparing $q$-expansions to determine the actual coefficients). Because $\ord_{\hat{e}_1}(E_{1,\chi})<1$, and $\tU_7(x^j)$ is a legitimate function on $X_0(49)$, it follows that $\tU_7(x^j)$ can {\em not} have a pole at this point. Finally, while $\tU_7(x^j)$ clearly does have a pole at $Q$, it is easily moved to $\infty$ when we multiply by $(x-x_Q)$.
\end{proof}

From Proposition \ref{Prop:GlobalUxi}, we could write each $\tU_7(x^i)$ for $i\leq 6$ explicitly as a rational function in $x$ and $z$ over $K=\Q(\zeta_{42})$. Approximations of these functions will suffice for our purposes, but in order to give an approximation we must first be clear about how the global field is embedded into $\C_7$. Recall from Section \ref{Section:Intro} that the ideal $(7)$ factors in $K$ as $(\pi_1)^6(\pi_2)^6$, where
$$\pi_1=-\zeta^8 + \zeta^6 - \zeta^4 + \zeta \quad\text{and}\quad  \pi_2=\zeta^9 + \zeta^8 + \zeta^4 + \zeta^3 - \zeta - 1.$$
Therefore, completing $K$ at either prime ideal $(\pi_i)$ results in a degree $6$ totally ramified extension of $\Q_7$ for which $\pi_i$ is a uniformizer.
We call the resulting two complete fields $K_1$ and $K_2$, respectively. 

Over either $K_i$, we may consider the reduced affinoid $\bA\subseteq X_0(49)$ defined over $\Q_7$  by $v(x^2+7)=1$, or equivalently $v(z)=3/4$. Instead of stating our approximations for $\tU_7(x^i)$ in terms of individual coefficients, we will bound our error terms using the spectral norm on $\bA$, which is highly compatible with the sup norm on $\tD$ that was mentioned above.\footnote{Over $\hat{L}$, $\bA$ is the complement in $\cX$ of four residue classes, one of which is $\tD$. So for any $f$ which is holomorphic on $\tD\cup\bA$, we have $|f|_{sup}=||f||_{\bA}$.} In order to simplify things notationally, for any $f\in A_{K_i}(\bA)$, let $\vv_i(f)$ be twice the minimal $\pi_i$-adic valuation of $f$ over all $\C_7$-valued points of $\bA$ (so $||f||_{\bA}=7^{-\vv_i(f)/12}$). For either $i$, we then have $\vv_i(x)=6$, $\vv_i(z)=9$, and $\vv_i(y)=21$. The following proposition gives sufficiently precise approximation formulas for the $\tU_7(x^i)$, with the error bounded using the spectral norm on $\bA$ in this manner.

\begin{proposition}\label{Prop:Uxi-Approx1}
Approximations for the functions in Proposition \ref{Prop:GlobalUxi} over the field $K_1$ are as given below. We write $f\equiv g$, $\vv_1=a$, $\ee_1\geq b$ to mean that $\vv_1(f)=\vv_1(g)=a$ and $\vv_1(f-g)\geq b$.
\begin{align*}
\tfrac{y(x-x_Q)}{x}\cdot\tU_7(x)&\equiv 6z(x+\pi_1^3)^2,\quad &&\vv_1=21,\quad\ee_1\geq22 \\
\tfrac{y(x-x_Q)}{x}\cdot\tU_7(x^2)&\equiv xz(x+\pi_1^3)+5\pi_1^2x^2(x+\pi_1^3),\quad &&\vv_1=21,\quad\ee_1\geq23  \\
\tfrac{y(x-x_Q)}{x}\cdot\tU_7(x^3)&\equiv   2\pi_1 x^2z ,\quad &&\vv_1=23,\quad\ee_1\geq24 \\
\tfrac{y^2(x-x_Q)}{x}\cdot \tU_7(x^4)&\equiv  \pi_1 x^4(x^2+7)(x+\pi_1^3),\quad &&\vv_1=44,\quad\ee_1\geq 45 \\
\tfrac{y^2(x-x_Q)}{x}\cdot \tU_7(x^5)&\equiv 3x^5z(x+\pi_1^3)+\\
                                                                &\qquad\qquad 2\pi_1^2x^5(x+\pi_1^3)(x+4\pi_1^3),\quad && \vv_1=45,\quad\ee_1\geq47   \\
\tfrac{y^2(x-x_Q)}{x}\cdot \tU_7(x^6)&\equiv 4\pi_1^2 x^6(x+\pi_1^3),\qquad && \vv_1=46,\quad\ee_1\geq 47
\end{align*}
\end{proposition}
\begin{proof}
In each case, we simply compute $\vv_1$ of all individual terms in the particular polynomial in $x$ and $z$. With the exception of $\tU_7(x^2)$ and $\tU_7(x^5)$, we then keep only those terms for which $\vv_1$ was minimal. Note that we have also taken ``first order'' approximations of the coefficients.  For $\tU_7(x^2)$ and $\tU_7(x^5)$, we also hold onto a second level of terms. In the later stages of our proof it will become evident why the extra level of precision was necessary in these two cases, namely because we are forced to do one column operation on the matrix representing $\tU_7$ while maintaining the approximation.
\end{proof}

By precisely the same reasoning, then, we derive the analogous approximation formulas for $\tU_7(x^i)$ over $K_2$.
\begin{proposition}\label{Prop:Uxi-Approx2}
Approximations for the functions in Proposition \ref{Prop:GlobalUxi} over the field $K_2$ are as given below. We write $f\equiv g$, $\vv_2=a$, $\ee_2\geq b$ to mean that $\vv_2(f)=\vv_2(g)=a$ and $\vv_2(f-g)\geq b$.
\begin{align*}
\tfrac{y(x-x_Q)}{x}\cdot\tU_7(x)&\equiv 3\pi_2 z(x+\pi_2^3)^2 ,\quad &&\vv_2=23,\quad\ee_2\geq 24\\
\tfrac{y(x-x_Q)}{x}\cdot\tU_7(x^2)&\equiv 2\pi_2 xz(x+\pi_2^3)+5\pi_2^3 x^2(x+\pi_2^3)  ,\quad &&\vv_2=23,\quad\ee_2\geq25 \\
\tfrac{y(x-x_Q)}{x}\cdot\tU_7(x^3)&\equiv 3\pi_2^2 x^2z,\quad &&\vv_2=25,\quad\ee_2\geq 26\\
\tfrac{y^2(x-x_Q)}{x}\cdot \tU_7(x^4)&\equiv 2\pi_2^2 x^4(x^2+7)(x+\pi_2^3),\quad &&\vv_2=46,\quad\ee_2\geq 47\\
\tfrac{y^2(x-x_Q)}{x}\cdot \tU_7(x^5)&\equiv \pi_2 x^5 z(x+\pi_2^3)+\\
                                                                &\qquad\qquad 5\pi_2^3 x^5(x+\pi_2^3)(x+4\pi_2^3) ,\quad &&\vv_2=47,\quad\ee_2\geq 49 \\
\tfrac{y^2(x-x_Q)}{x}\cdot \tU_7(x^6)&\equiv  6\pi_2^3 x^6(x+\pi_2^3)      ,\quad &&\vv_2=48,\quad\ee_2\geq 49
\end{align*}
\end{proposition}

\subsection{Calculation of $\tU_7(s^i)$ for $i=1,\dots,7$}

Now that we have approximations for $\tU_7(x^j)$, $j=1,\dots,6$, we can use these to generate approximations for $\tU_7(s^i)$. Once again, the main idea here is to write each $s^i$ in the form,
$$s^i=g_{i,6}(y)x^6+g_{i,5}(y)x^5+\cdots+g_{i,1}(y)x+g_{i,0}(y),$$
which we know we can do since the function field of $X_0(49)$ over $K$ is a degree $7$ extension of $K(y)$. 
We then use Equation \ref{Eq:ColemanTwistedTrick} to conclude that
$$\tU_7(s^i)=g_{i,6}(t)\tU_7(x^6)+g_{i,5}(t)\tU_7(x^5)+\cdots+g_{i,1}(t)\tU_7(x)+g_{i,0}(t).$$
Finally, we approximate the $g_{i,j}(t)$ using the fact that $\vv_1(t)=\vv_2(t)=3$, and combine these approximations with the ones from Propositions \ref{Prop:Uxi-Approx1} and \ref{Prop:Uxi-Approx2} to obtain approximations for the $\tU_7(s^i)$ with respect to either embedding. To simplify matters slightly, we initially deal with $t^{-i}$ rather than $s^i$. These differ by a scalar, and $t$ is defined over $\Q_7$. 

\begin{proposition}\label{Prop:g_ij(t)-Approx}
Let $\bA\subseteq X_0(49)$ be the affinoid over $\Q_7$ defined by {\mbox{$v(x^2-7)=1$,}} as above. Write
$$t^{-i}=g_{i,6}(y)x^6+g_{i,5}(y)x^5+\cdots+g_{i,1}(y)x+g_{i,0}(y),\qquad 1\leq i\leq 7.$$
Then $t^{7i}g_{i,j}(t)$ is a polynomial in $t$ of degree less than $7i$. Each of these polynomials has a unique dominant term on $\bA$ which is given in the following table.

\begin{center}
\begin{tabular}{c||ccccccc}
 $t^{-i}$  & $t^{7i}g_{i,6}(t)$ & $t^{7i}g_{i,5}(t)$ & $t^{7i}g_{i,4}(t)$ & $t^{7i}g_{i,3}(t)$ & $t^{7i}g_{i,2}(t)$ & $t^{7i}g_{i,1}(t)$ & $t^{7i}g_{i,0}(t)$\\ \hline
 $t^{-1}$ & $7t^5$ & $5\cdot 7t^5$ & $2\cdot 7^2 t^5$ & $t^6$ &  $7t^6 $ & $3\cdot7 t^6$ &$ 5\cdot7^2 t^6$\\
 $t^{-2}$ & $t^{12}$ & $2\cdot 7t^{12}$ & $6\cdot 7t^{12}$ & $5\cdot7^2t^{12}$ & $2\cdot7t^{13}$ & $3\cdot7t^{13}$ & $4\cdot7^2t^{13}$\\
 $t^{-3}$ & $6\cdot7^2t^{18}$ & $2\cdot 7t^{19}$ & $3\cdot 7t^{19}$ & $6\cdot7^2t^{19}$ & $t^{20}$ & $2\cdot7t^{20}$ & $6\cdot7t^{20}$\\
 $t^{-4}$ & $2\cdot7^2t^{25}$ & $t^{26}$ & $3\cdot7t^{26}$ & $2\cdot 7t^{26}$ & $5\cdot7^2t^{26}$ & $3\cdot 7t^{27}$ & $7t^{27}$\\
 $t^{-5}$ & $2\cdot7t^{32}$ & $5\cdot7^2t^{32}$ & $3\cdot 7 t^{33}$ & $7t^{33}$ & $2\cdot7^3t^{32}$ & $t^{34}$ & $3\cdot7t^{34}$\\
 $t^{-6}$ & $6\cdot7t^{39}$ & $7^2t^{39}$ & $t^{40}$ & $4\cdot 7 t^{40}$ & $5\cdot 7 t^{40}$ & $3\cdot 7^2t^{40}$ & $4\cdot7t^{41}$\\
 $t^{-7}$ & $4\cdot7t^{46}$ & $5\cdot 7t^{46}$ & $6\cdot7^2t^{46}$ & $4\cdot7t^{47}$ & $6\cdot7t^{47}$ & $3\cdot7^2t^{47}$ & $t^{48}$
\end{tabular}
\end{center}

\end{proposition}

\begin{proof}
It follows directly from Equation \ref{Eq:X49model} that $y^2/x$ is a polynomial in $x$ and $y$. Indeed, one can simply solve the equation to get
$$y^2/x=x^6 + 7x^5 + 21x^4 + 49x^3 + (7y + 147)x^2 + (35y + 343)x + 49y + 343.$$
Clearly, this equation can be used to write $(y^2/x)^i$ as a polynomial in $x$ and $y$. The degree of the resulting polynomial in $x$ may initially be quite large. However, the function which it represents can then be brought into ``standard form'', i.e. written as a polynomial of degree at most $6$ in $x$, by repeatedly substituting 
$$x^7=-7x^6-21x^5 -49x^4 -154x^3 -378x^2 -392x + y^2,$$
which is just the same equation in a different form.

Once we have a method for writing $(y^2/x)^i$ in standard form, it carries over directly to $t^{-i}$, since
$$t^{-i}=(y/x^4)^i=y^{-7i}(y^2/x)^{4i}.$$
We simply write
$$(y^2/x)^{4i}=h_{i,6}(y)x^6+h_{i,5}(y)x^5+\cdots+h_{i,1}(y)x+h_{i,0}(y),$$
where each $h_{i,j}$ is a polynomial (whose degree, one sees, is less than $7i$). Then we are done, since we now have $g_{i,j}(y)=y^{-7i}h_{i,j}(y)$.  To determine the dominant term of $g_{i,j}(t)=t^{-7i}h_{i,j}(t)$, we may equivalently determine the dominant term of $h_{i,j}(t)$. This is a very straightforward calculation using $\vv_1(t)=\vv_2(t)=3$.
\end{proof}


We are finally in a position now to write down approximation formulas for $\tU_7(s^i)$. For this, we must work over the $7$-adic field $\hat{L}$, where $L$ is a number field containing both $K$ and a root $\alpha$ of $x^4+7$. Note that since $\hat{L}$ contains either $K_1$ or $K_2$, both $\vv_1$ and $\vv_2$ extend in a natural way through the spectral norm on $\bA\otimes \hat{L}$.

\begin{proposition}\label{Prop:Explicit_Usi1}
Approximations for $\tU_7(s^i)$ for $1\leq i\leq7$ over $\hat{L}\supseteq K_1$ are as follows.
\begin{align*}
\tU_7(s^1)&\equiv 2\alpha\pi_1 z/(x(x+\pi_1^3)),\quad &&\vv_1= 2,\quad \ee_1\geq 3\\
\tU_7(s^2)&\equiv 4\alpha^2\pi_1^2/x,\quad &&\vv_1=4,\quad \ee_1\geq 5\\
\tU_7(s^3)&\equiv \alpha^3z/x^2+5\alpha^3\pi_1^2/x,\quad &&\vv_1=6,\quad \ee_1\geq 8\\
\tU_7(s^4)&\equiv 3\alpha^4z/x^2+2\alpha^4\pi_1^2(x+4\pi_1^3)/x^2,\quad &&\vv_1=9,\quad \ee_1\geq11\\
\tU_7(s^5)&\equiv 6\alpha^5z(x+\pi_1^3)/x^3,\quad &&\vv_1=12,\quad \ee_1\geq13\\
\tU_7(s^6)&\equiv \alpha^6\pi_1(x^2+7)/x^3,\quad &&\vv_1=14,\quad \ee_1\geq15\\
\tU_7(s^7)&\equiv \alpha^7/t,\quad &&\vv_1=18,\quad \ee_1\geq19
\end{align*}
\end{proposition}

\begin{proof}
Taking into account that $\vv_1(x)=\vv_1(x-x_Q)=6$ and $\vv_1(y)=21$, and applying Proposition \ref{Prop:Uxi-Approx1}, we see that
$$\left(\vv_1\big(\tU_7(x^6)\big),\dots,\vv_1\big(\tU_7(x)\big),\vv_1(1)\right)=(4,3,2,2,0,0,0).$$
Then, from Proposition \ref{Prop:g_ij(t)-Approx}, we can compute $\vv_1$ of each $g_{i,j}(t)$. These values are collected for convenience in the following matrix. Note that the ordering of the terms is consistent with the table from that proposition.
$$\left[\begin{matrix}
6 & 6 & 18 & -3 & 9 & 9 & 21\\
-6 & 6 & 6 & 18 & 9 & 9 & 21\\
15 & 6 & 6 & 18 & -3 & 9 & 9\\
15 & -6 & 6 & 6 & 18 & 9 & 9\\
3 & 15 & 6 & 6 & 27 & -3 & 9\\
3 & 15 & -6 & 6 & 6 & 18 & 9\\
3 & 3 & 15 & 6 & 6 & 18 & -3
\end{matrix}\right]$$
Now, adding the entries of the $i\nth$ row to the $\vv_1$ values of $\tU_7(x^j)$ (given above), we are able to determine the dominant term(s) in our approximation for $\tU_7(t^{-i})$. Then we scale by $\alpha^i$ to obtain an approximation for $\tU_7(s^i)$. We will do $\tU_7(s^1)$ in great detail, and then give only the essential information for the $i>1$ cases as they are very similar.

For $\tU_7(t^{-1})$, we look at the first row of the matrix, and see that the unique dominant term will be $g_{1,3}(t)\cdot\tU_7(x^3)$ for which $\vv_1=-1$ (for all other terms, $\vv_1\geq9$). So in order to approximate $\tU_7(t^{-1})$ we multiply the approximations for $g_{1,3}(t)$ and $\tU_7(x^3)$ from Propositions \ref{Prop:g_ij(t)-Approx} and \ref{Prop:Uxi-Approx1}. 
\begin{align*}
\tU_7(t^{-1})&\equiv \frac{1}{t}\cdot\frac{x}{y(x-x_Q)}\cdot 2\pi_1x^2z,\quad &&\vv_1=-1,\quad \ee_1=0\\
                     &\equiv 2\pi_1 z/(x(x+\pi_1^3)),\quad && \vv_1=-1,\quad \ee_1=0
\end{align*}
Here we have used the facts that $t=x^4/y$ and $\vv_1(x_Q+\pi_1^3)=8$. Finally, we multiply through by $\alpha^1$, since $s=\alpha t^{-1}$ and arrive at the stated approximation for $\tU_7(s^1)$. We summarize this process for $i>1$ in what follows.

\begin{align*}
\tU_7(t^{-2})&\equiv g_{2,6}(t)\cdot\tU_7(x^6),\quad &&\vv_1=-2,\quad \ee_1\geq 8\\
                     &\equiv \frac{1}{t^2}\cdot\frac{x}{y^2(x-x_Q)}\cdot 4\pi_1^2x^6(x+\pi_1^3),\quad&& \vv_1=-2,\quad\ee_1\geq-1\\
                     &\equiv 4\pi_1^2/x,\quad &&\vv_1=-2,\quad\ee_1\geq -1\\
\tU_7(t^{-3})&\equiv g_{3,2}(t)\cdot\tU_7(x^2),\quad && \vv_1=-3,\quad\ee_1\geq 8\\
                     &\equiv \frac{1}{t}\cdot \frac{x}{y(x-x_Q)}\left[xz(x+\pi_1^3)+5\pi_1^2x^2(x+\pi_1^3)\right],\quad && \vv_1=-3,\quad\ee_1\geq -1\\
                     &\equiv z/x^2+5\pi_1^2/x,\quad && \vv_1=-3,\quad \ee_1\geq-1\\
\tU_7(t^{-4})&\equiv g_{4,5}(t)\cdot \tU_7(x^5),\quad && \vv_1=-3,\quad \ee_1\geq 8\\
                     &\equiv \frac{1}{t^2}\cdot\frac{x}{y^2(x-x_Q)}\cdot\big[3x^5z(x+\pi_1^3)+\\
                     &\qquad\qquad\qquad\quad\qquad 2\pi_1^2x^5(x+\pi_1^3)(x+4\pi_1^3)\big],\quad && \vv_1=-3,\quad \ee_1\geq -1\\
                     &\equiv 3z/x^2+2\pi_1^2(x+4\pi_1^3)/x^2,\quad && \vv_1=-3,\quad \ee_1\geq -1\\
\tU_7(t^{-5})&\equiv g_{5,1}(t)\cdot\tU_7(x^1),\qquad && \vv_1=-3,\quad \ee_1\geq 7\\
                     &\equiv \frac{1}{t}\cdot\frac{x}{y(x-x_Q)}\cdot 6z(x+\pi_1^3)^2,\quad && \vv_1=-3,\quad\ee_1\geq-2\\
                     &\equiv 6z(x+\pi_1^3)/x^3,\quad && \vv_1=-3,\quad\ee_1\geq-2\\
\tU_7(t^{-6})&\equiv g_{6,4}(t)\cdot\tU_7(x^4),\qquad &&\vv_1=-4,\quad \ee_1\geq 6\\
                     &\equiv \frac{1}{t^2}\cdot\frac{x}{y^2(x-x_Q)}\cdot \pi_1 x^4(x^2+7)(x+\pi_1^3),\qquad &&\vv_1=-4,\quad \ee_1\geq -3\\
                     &\equiv \pi_1(x^2+7)/x^3,\qquad &&\vv_1=-4,\quad \ee_1\geq -3\\
\tU_7(t^{-7})&\equiv g_{7,0}(t)\cdot 1 \equiv 1/t,\qquad &&\vv_1=-3,\quad \ee_1\geq 0
\end{align*}

\end{proof}

By precisely the same reasoning, we arrive at the following approximation formulas in the case of a Type 2 embedding.

\begin{proposition}\label{Prop:Explicit_Usi2}
Approximations for $\tU_7(s^i)$ for $1\leq i\leq7$ over $\hat{L}\supseteq K_2$ are as follows.
\begin{align*}
\tU_7(s^1)&\equiv 3\alpha\pi_2^2z/(x(x+\pi_2^3)),\quad &&\vv_2= 4,\quad \ee_2\geq 5\\
\tU_7(s^2)&\equiv 6\alpha^2\pi_2^3/x,\quad &&\vv_2=6,\quad \ee_2\geq 7\\
\tU_7(s^3)&\equiv 2\alpha^3\pi_2 z/x^2+5\alpha^3\pi_2^3/x,\quad &&\vv_2=8,\quad \ee_2\geq 10\\
\tU_7(s^4)&\equiv \alpha^4 \pi_2 z/x^2+5\alpha^4\pi_2^3(x+4\pi_2^3)/x^2,\quad &&\vv_2=11,\quad \ee_2\geq13\\
\tU_7(s^5)&\equiv 3\alpha^5\pi_2z(x+\pi_2^3)/x^3,\quad &&\vv_2=14,\quad \ee_2\geq15\\
\tU_7(s^6)&\equiv 2\alpha^6\pi_2^2(x^2+7)/x^3,\quad &&\vv_2=16,\quad \ee_2\geq17\\
\tU_7(s^7)&\equiv \alpha^7/t,\quad &&\vv_2=18,\quad \ee_2\geq19
\end{align*}
\end{proposition}

\subsection{Recurrence Relation and the Final Matrix}

Now that we have approximations for $\tU_7(s^i)$, $i=1,\dots,7$, this can be extended to all $i\geq 1$ by means of a $7^{\nth}$ order linear recurrence relation with coefficients in $\hat{L}(t)$, as in \cite[\S4]{kilford-5slopes}. The reason for this is essentially the same key fact which was used in the previous section, that inside the function field of $X_0(49)$, $s$ is algebraic of degree $7$ over $\hat{L}(y)$. So for fixed rational functions, $g_0(y),g_1(y),\dots,g_6(y)$, we have
$$s^{i+7}=g_6(y)s^{i+6}+g_5(y)s^{i+5}+\cdots+g_1(y)s^{i+1}+g_0(y)s^i.$$
Therefore, applying Equation \ref{Eq:ColemanTwistedTrick} as we have done before, we have
$$\tU_7(s^{i+7})=g_6(t)\tU_7(s^{i+6})+g_5(t)\tU_7(s^{i+5})+\cdots+g_1(t)\tU_7(s^{i+1})+g_0(t)\tU_7(s^i).$$

The only practical difficulty could be in finding the coefficient functions. As we have already seen, however, it is straightforward to write any power of $x$ in the basis $\{x^6,x^5,\dots,1\}$ over $K(y)$ by repeatedly applying the identity,
$$x^7=-7x^6-21x^5 -49x^4 -154x^3 -378x^2 -392x + y^2.$$
Also, we know that $t=x^4/y$. So one strategy is to write each $x^{4i}$ for $0\leq i\leq 7$ in the basis $\{x^6,x^5,\dots,x,1\}$, and then use linear algebra to solve for $x^{28}$ as a linear combination of the $7$ linearly independent vectors, $1,x^4,x^8,\dots,x^{24}$. We find that
$$x^{28}=h_6(y)x^{24}+h_5(y)x^{20}+\cdots+h_1(y)x^{4}+h_0(y),$$
for the following polynomial coefficient functions.
\begin{align*}
h_6(y)&= -28y - 49\\
h_5(y)&=-322y^2 - 1372y - 2401\\
h_4(y)&= -1904y^3 - 15778y^2 - 67228y - 117649\\
h_3(y)&= -5915y^4 - 93296y^3 - 773122y^2 - 3294172y - 5764801\\
h_2(y)&= -8624y^5 - 289835y^4 - 4571504y^3 \\
           &\qquad\qquad\qquad - 37882978y^2 - 161414428y - 282475249\\
h_1(y)&= -4018y^6 - 422576y^5 - 14201915y^4\\
           &\qquad\qquad\qquad - 224003696y^3 - 1856265922y^2 - 7909306972y - 13841287201\\
h_0(y)&=y^8
\end{align*}
Substituting $(yt)^i$ for each $x^{4i}$, it follows that
$$t^{-7}=-h_1(y)y^{-7}t^{-6}-h_2(y)y^{-6}t^{-5}-\cdots-h_6(y)y^{-2}t^{-1}+y^{-1}.$$
And finally we substitute $t=\alpha/s$ and apply Equation \ref{Eq:ColemanTwistedTrick} to obtain the recurrence relation for $\tU_7(s^i)$.
\begin{equation}\label{Eq:s-recurrence}
\tU_7(s^{i+7})=-\alpha h_1(t)t^{-7}\tU_7(s^{i+6})-\cdots-\alpha^6 h_6(t)t^{-2}\tU_7(s^{i+1})+\alpha^7t^{-1}\tU_7(s^i)
\end{equation}
Putting the above recurrence relation for $\tU_7(s^i)$ together with the explicit approximations for $\tU_7(s^i)$ when $1\leq i\leq 7$ from Propositions \ref{Prop:Explicit_Usi1} and \ref{Prop:Explicit_Usi2}, we are now in a position to write down approximation formulas for $\tU_7(s^i)$ for all $i$. These are captured most succinctly by the following proposition.


\begin{proposition}\label{Prop:Recurrence_Usi}
Suppose that $1\leq i\leq 7$ and $j\geq0$. Let $V_{1,i}=\vv_1(\tU_7(s^i))$ and $V_{2,i}=\vv_2(\tU_7(s^i))$. Then
$$\tU_7(s^{7j+i})\equiv \alpha^{6j}s^j\tU_7(s^i),\quad \vv_1=18j+V_{1,i},\quad \ee_1\geq 18j+2+V_{1,i},$$
and the analogous approximation formula holds for $\vv_2$.
\end{proposition}

\begin{proof}
This is straightforward to prove by induction on $j$. The key is to compute the sizes of the coefficient functions, $\alpha^kh_k(t)t^{k-8}$, $1\leq k\leq 6$, in Equation \ref{Eq:s-recurrence}. These functions end up being so small on $\bA$ (regardless of the embedding), that only the $\alpha^7 t^{-1}\tU_7(s^i)$ term in the recurrence relation ends up being non-negligible. In particular, using the facts that $\vv_1(t)=\vv_1(\alpha)=3$, and $\vv_1(7)=12$, we obtain the following.

\begin{table}[h!]
\begin{center}
\begin{tabular}{| c || c | c | c | c | c | c |} \hline
 $k$ & 1 & 2 & 3 & 4 & 5  & 6  \\ \hline
 $\vv_1(\alpha^kh_k(t)t^{k-8})$ & 24 & 27 & 18 & 21 & 24 & 27  \\ \hline
\end{tabular}
\end{center}
\end{table}
On the other hand, it is immediate that $\vv_1(\alpha^7t^{-1})$ is just $18$. So under the assumption of the inductive hypothesis (which forces $\tU_7(s^{i+4})$ to be much smaller than $\tU_7(s^i)$), the first six terms in the recurrence relation are always negligible. Thus, applying the recurrence relation finishes the inductive argument. The argument for $\vv_2$ is identical.
\end{proof}

\section{Proof of the Main Theorem}\label{Section:MainTheorem}

Now we are ready to prove a series of slope formulas. First we prove a formula for the slopes of all weight $1$ overconvergent forms in $M_{1,\chi}(49)$. Then we extend to all weights using powers of the Eisenstein series $E_{1,\tau}$. Finally, we conclude by applying results of Coleman and Cohen-Oesterl\'e to determine the slopes of all classical forms with a specified character.

The main idea in the proof of the first result is to represent $\tU_7$ acting on $\cM_0$ as an infinite matrix by working in the ``basis,'' $\{s,s^2,s^3,\dots\}$. Then we show that the matrix has a characteristic series, and compute the valuations of its coefficients. We will see that these coefficients, $c_j$, converge to $0$ so quickly that in fact $|c_{j+1}/c_j|$ forms a strictly decreasing null sequence. Once this is established, it is an easy lemma to show that each Newton slope of the characteristic series corresponds to a one-dimensional eigenspace, and that no other overconvergent eigenforms with finite slope can exist.

\begin{theorem}\label{Theorem:MainTheorem-Weight1}
Fix a primitive $42^{\text{nd}}$ root of unity, $\zeta$, and let $\chi$ be the Dirichlet character of conductor $49$ defined by $\chi(3)=\zeta$. Let $L$ be a number field containing $K=\Q(\zeta)$ and a root $\alpha$ of $x^4+7$. For any Type 1 embedding of $L$ into $\C_7$, the finite slopes of $U_7$ acting on $M_{1,\chi}(49)$ are given by
$$\left\{\tfrac{1}{6}\cdot\left\lfloor\tfrac{9i}{7}\right\rfloor: i \in \N\right\}.$$
For any Type 2 embedding, the finite slopes are
$$\left\{\tfrac{1}{6}\cdot\left\lfloor\tfrac{9i+6}{7}\right\rfloor: i \in \N\right\}.$$
In either case, the eigenspaces are all one-dimensional and defined over $\hat{L}$.
\end{theorem}

\begin{proof}
We will make the argument for Type 1 only, as the proof for Type 2 is identical. First we fix some notation. Since $\tU_7(s^j)$ may be viewed as a holomorphic function on the unit disk which vanishes at the origin, we may write it uniquely as
$$\tU_7(s^j)=\sum_{i=1}^{\infty}a_{i\,j}s^i.$$
Note that each of these functions has finite sup norm, given explicitly by Propositions \ref{Prop:Explicit_Usi1} and \ref{Prop:Recurrence_Usi}, and that this determines the minimal valuation of the coefficients $a_{i\,j}$. Philosophically, we think of $M=(a_{i\,j})$ as the matrix representing $\tU_7$, and hence we call the $\tU_7(s^j)$ the ``column functions.''

Now, let $M_n$ be the $n\times n$ truncation of $M$, i.e., $M_n=(a_{i\,j})_{1\leq i,j\leq n}$. We define the characteristic polynomial of $M_n$ to be 
$$f_n(\lambda)=(-1)^n\lambda^n\det\left(M_n-\tfrac{1}{\lambda}\cdot I_n\right).$$
Clearly, $\lambda$ is a nonzero eigenvalue of $M_n$ if and only if $1/\lambda$ is a root of $f_n$. We will show that these polynomials converge to a characteristic series for $M$. The key is to interpret the coefficients in terms of the classical matrix invariants. In particular, let 
$$f_n(\lambda)=1-c_{n\,1}\lambda+c_{n\,2}\lambda^2-\cdots+(-1)^n c_{n\,n}\lambda^n.$$
Then $c_{n\,1}$ is simply the trace of $M_n$ and $c_{n\,n}$ is the determinant. More generally, $c_{n\,j}$ is the sum of the determinants of all principal $j\times j$ minors of $M_n$, i.e., those obtained from $M_n$ by deleting any $(n-j)$ rows and then the {\em same} $(n-j)$ columns. Clearly, since the sup norms of the column functions form a decreasing null sequence, each $(c_{n\,j})_{n\geq 1}$ is a Cauchy and thus convergent sequence. Indeed, for a fixed $j>0$ and any $m>n\geq j$, $c_{m\,j}-c_{n\,j}$ is the sum of the determinants of all principal $j\times j$ minors of $M_m$ which retain at least one column whose index is greater than $n$. Thus, using the fact that all of the coefficients of $M$ are integral, we can bound $|c_{m\,j}-c_{n\,j}|$ with the sup norm of the $(n+1)^{\text{st}}$ column function. For notation, let $c_j=\lim_{n\to\infty} c_{n\,j}$. Then we define the characteristic series of $M$ to be 
$$f(\lambda)=1+\sum (-1)^jc_j\lambda^j.$$

Next, viewing each $c_{n\,j}$ as the sum of principal minors, we show that in fact $\det(M_j)$ is always the leading term by computing its valuation explicitly. To do this, we consider the reductions (after finitely many elementary column operations) of the column functions, on the model $\cX$ for $X_0(49)$ given in Equation \ref{Eq:X49Reduction}. We may assume without loss of generality that
$$\alpha^2=2\zeta^{11} - 2\zeta^9 - 2\zeta^8 - 2\zeta^4 + 2\zeta + 1,$$
and hence $\vv_1(\alpha^2+\pi_1^3)=8$. Referring back to Proposition \ref{Prop:Explicit_Usi1}, we can  subtract $3\alpha\cdot \tU_7(s^3)$ from $\tU_7(s^4)$ and then divide each column by an appropriate scalar, to obtain the following reductions for the first seven column functions:
 $$\tfrac{Z}{X(X-1)},\tfrac{1}{X},\tfrac{Z}{X^2},\tfrac{X-1}{X^2},\tfrac{Z(X-1)}{X^3},\tfrac{X^2-1}{X^3},\tfrac{Z(X^2-1)}{X^4}.$$
Similarly, if we subtract $3\alpha\cdot\tU_7(s^{10})$ from $\tU_7(s^{11})$ and scale appropriately, then by  Proposition \ref{Prop:Recurrence_Usi} the reductions of the {\em next} seven column functions will simply be the product of these first seven with an extra $Z(X^2-1)/X^4$, and so on. We would like to show that the expansions of the first $j$ of these reduced functions in $\hat{\sO}_{\ol{\cX},\infty}$ are always linearly independent up through the $s^j$ term. This follows easily from the divisors of the reduced functions on $\ol{\cX}$. Indeed, using $(X,Z)$ coordinates for points, the first seven reduced column functions have divisors:
\begin{align*}
(Z/(X(X-1)))&=(\infty)+(-1,0)-(0,0)-(1,0)\\
(1/X)&=2(\infty)-2(0,0)\\
(Z/X^2)&=(1,0)+(-1,0)+(\infty)-3(0,0)\\
((X-1)/X^2)&=2(1,0)+2(\infty)-4(0,0)\\
(Z(X-1)/X^3)&=3(1,0)+(-1,0)+(\infty)-5(0,0)\\
((X^2-1)/X^3)&=2(1,0)+2(-1,0)+2(\infty)-6(0,0)\\
(Z(X^2-1)/X^4)&=3(1,0)+3(-1,0)+(\infty)-7(0,0).
\end{align*}
Then, each time we multiply by $Z(X^2-1)/X^4$ to obtain the next seven functions, we add $3(1,0)+3(-1,0)+(\infty)-7(0,0)$ to the divisors. From the poles at $(0,0)$ alone, it is immediate that the first $j$ functions are always linearly independent. But this is not enough. We need to show that in fact no nontrivial linear combination could even be a function which vanishes $j+1$ times at $\infty$. Suppose we had such a linear combination. At worst, the function would be in $L((1,0)+j(0,0))$. So it would have to have divisor exactly $(j+1)(\infty)-(1,0)-j(0,0)$. If $j$ were odd, we could then use $(X)=2(0,0)-2(\infty)$ to produce a function with divisor $(0,0)-(1,0)$. If $j$ were even, we could use $X$ to produce a function with divisor $(\infty)-(1,0)$. Either is a contradiction, as the curve is not rational. So the expansions in $\hat{\sO}_{\ol{\cX},\infty}$ must be linearly independent up through the $s^j$ term. Therefore, passing through the isomorphism with $\ol{A_{\hat{L}}(\tD)}$, and taking into account the scaling factors, we have shown that
$$v(\det(M_j))=\sum_{i=1}^j\tfrac{1}{6}\cdot\left\lfloor\tfrac{9i}{7}\right\rfloor.$$

Finally, since $\vv_1$ of any later column function must exceed $\vv_1$ of any of the first $j$ column functions by at least $2$, and each of the elementary column operations which were performed on $M_j$ only increased $\vv_1$ of that column by $1$, it follows that $\det(M_j)$ is indeed the unique dominant term in the convergent sum of principal $j\times j$ minors defining $c_{n\,j}$. Therefore the above formula for $v(\det(M_j))$ is in fact a formula for $v(c_j)$. Having established that $|c_{j+1}/c_j|$ is a strictly decreasing null sequence, the claims about slopes and eigenspaces easily follow.
\end{proof}

\begin{theorem}\label{Theorem:MainTheorem-Overconvergent}
Let $k\in\mathbb{N}$ be arbitrary. Fix a primitive $42^{\text{nd}}$ root of unity, $\zeta$, and let $\chi$ be the Dirichlet character of conductor $49$ defined by $\chi(3)=\zeta$.
Let $L$ be a number field containing $K=\Q(\zeta)$ and a root $\alpha$ of $x^4+7$. For any Type 1 embedding of $L$ into $\C_7$, the finite slopes of $U_7$ acting on $M_{k,\chi^{7k-6}}(49)$ are given by
$$\left\{\tfrac{1}{6}\cdot\left\lfloor\tfrac{9i}{7}\right\rfloor: i \in \N\right\}.$$
For any Type 2 embedding, the finite slopes of $U_7$ acting on $M_{k,\chi^{8-7k}}(49)$ are given by
$$\left\{\tfrac{1}{6}\cdot\left\lfloor\tfrac{9i+6}{7}\right\rfloor: i \in \N\right\}.$$
In either case, the eigenspaces are all one-dimensional and defined over $\hat{L}$.
\end{theorem}

\begin{proof}
From the definition of $\tU_7$ (see Equation \ref{Def-TwistedU7}), we see that the infinite matrix representing $\tU_7$ on this weight $k$ space is obtained from the infinite matrix in the previous theorem by simply multiplying each column function by the {\em same} function. (That was the whole point of working with $\tU_7$ instead of the true pullback of $U_7$ to $\cM_0$.)
So the key to proving this theorem is to choose the auxiliary character $\tau$ appropriately in both cases.
In particular, if we choose it so that $E_{1,\tau}/V(E_{1,\tau})$ is a holomorphic function with sup norm $1$ on $\tD$ whose reduction in
$$\ol{A_{\hat{L}}(\tD)}\cong\hat{\sO}_{\ol{\cX},P}$$
does not vanish at $P$, the same proof will essentially goes through verbatim.

First, we apply Lemma \ref{Lemma:EllipticPoints} to obtain the following explicit formula for the extra weight factor:
$$\frac{E_{1,\tau}}{V(E_{1,\tau})}=(\beta+2)\cdot\frac{y-(3\beta-8)^{-1}x^4}{y-(3\beta-8)}\cdot\frac{z-(\beta-\frac{3}{2})x-2\beta+3}{z+(\frac{2}{7}\beta+\frac{1}{7})x(x+\frac{7}{2})}.$$
(The lemma implies that the two divisors agree, and then $q$-expansions verify that the constant is correct.) 

In the Type 1 case, we choose $\tau$ by setting $\tau(3)=\beta=\zeta^7$, which of course implies that $\chi\tau^{k-1}=\chi^{7k-6}$. It is easy to check (globally) that
\begin{align*}
v_{\pi_1}(t(e_\beta))&=v_{\pi_1}(3\beta-8)=12\\
v_{\pi_1}(x(\hat{e}_\zeta))&=v_{\pi_1}(3\zeta^7-1)=6
\end{align*}
So with this type of embedding into $\C_7$, both $E_{1,\chi}$ and $E_{1,\tau}$ are non-vanishing on $W_1(49)$ (i.e., we have $\delta_\zeta=\delta_\tau=0$). In Equation \ref{Def-TwistedU7}, then, we have $d_k=0$ and there is no holomorphicity factor to worry about. Moreover, if we do a valuation analysis on the above expression, we find that on $\bA$ (and over $K_1$) we have
\begin{align*}
\frac{E_{1,\tau}}{V(E_{1,\tau})}&\equiv (\beta+2)\cdot\frac{-(3\beta-8)^{-1}x^4}{y}\cdot\frac{z}{(\frac{2}{7}\beta+\frac{1}{7})x^2}\\
     &\equiv\frac{x^2z}{y}\\
     &\equiv\frac{x^2z}{z(x^2+7)}\equiv\frac{x^2}{x^2+7}\qquad \vv_1= 0,\quad \ee_1\geq 3.
\end{align*}
This function reduces to $X^2/(X^2-1)$ on the good reduction model $\cX$, and in particular is holomorphic and non-vanishing at $P$ (the infinite point).

The situation is very similar with the second embedding. This time we set $\tau(3)=\beta=\zeta^{-7}$, so that $\chi\tau^{k-1}=\chi^{8-7k}$. While $x(\hat{e}_\zeta)$ is now a unit (so $\delta_\zeta=1$ and $E_{1,\chi}$ has a zero on $W_1(49)$), the different choice of $\tau$ guarantees that once again $E_{1,\tau}$ will {\em not} vanish on $W_1(49)$. Thus, $\delta_\beta=0$ and we do not have to include the extra holomorphicity factor in $\tU_7$. The valuation analysis for $E_{1,\tau}/V(E_{1,\tau})$ on $\bA$ (and over $K_2$) is essentially the same and we find that
$$\frac{E_{1,\tau}}{V(E_{1,\tau})}\equiv\frac{x^2}{x^2+7}\qquad\vv_2=0,\quad \ee_2\geq 3.$$

So in both cases, the weight factor, $f_\tau:=E_{1,\tau}/V(E_{1,\tau})$, has sup norm $1$ and reduces to a function $\ol{f}_\tau$ on $\ol{\cX}$ which is holomorphic and non-vanishing at $P$. In going from weight $1$ to weight $k$ then, the column functions in the infinite matrix for $\tU_7$ are all multiplied by the same function $(f_\tau)^{k-1}$. Thus, the sup norms of all the column functions are unchanged. Moreover, after performing the exact same elementary column operations, and scaling by the exact same constants, the first $j$ column functions will each reduce to $(\ol{f}_\tau)^{k-1}$ times their old value. Now, suppose that some linear combination of the reductions of the first $j$ of these (adjusted) column functions was equal to a function $\ol{g}\in\hat{\sO}_{\ol{\cX},P}$ which vanished at $P$ with order $j+1$ or greater. Then the same linear combination of the reductions of the {\em original} first $j$ (adjusted) column functions would equal $\ol{g}\cdot(\ol{f}_\tau)^{1-k}$. But this function would still vanish $j+1$ times at $P$, because $\ol{f}_\tau$ was non-vanishing at $P$. Since we proved that the reductions of the first $j$ (adjusted) column functions in the weight $1$ matrix were independent up through the $s^j$ term (in the proof of Theorem \ref{Theorem:MainTheorem-Weight1}), this is a contradiction. Therefore, the same argument from the weight $1$ case can be used to show that $\det(M_j)$ is still the strictly leading term in the expansion for $c_j$, and of course its valuation has not changed. In short, although the characteristic series for $\tU_7$ has changed, the valuations of its coefficients have not. Thus, the slopes are the same, and the eigenspaces are once again one-dimensional.
\end{proof}

We are now ready to prove our main theorem regarding {\em classical} modular forms. In addition to the above theorem, we also apply here the theorem of Coleman that $U_p$ eigenforms of small slope are classical ( \cite[Theorem 1.1]{coleman-overconvergent}). The other key ingredient is the following special case of the well-known theorem of Cohen and Oesterl\'e.

\begin{theorem}[Cohen-Oesterl\'e]
Let $\chi$ be a primitive Dirichlet character of conductor $49$, and let $k$ be an integer greater than $1$. Then
$$\dim S_k(\Gamma_0(49),\chi)=\frac{14k-17}{3}+\epsilon(\chi(18)+\chi(30)),$$
where $\epsilon$ is $1/3$ if $k\equiv0\mod{3}$, $0$ if $k\equiv 1\mod{3}$, and $-1/3$ if $k\equiv 2\mod {3}$.
\end{theorem}

\begin{theorem}\label{Theorem:MainTheorem-Classical}
Let $k$ be an integer greater than $1$. Fix a primitive $42^{\text{nd}}$ root of unity, $\zeta$, and let $\chi$ be the Dirichlet character of conductor $49$ defined by $\chi(3)=\zeta$.

The classical space, $S_k(\Gamma_0(49),\chi^{7k-6})$, is diagonalized by the $U_7$ operator over the field $K_1(\alpha)$. The slopes of $U_7$ acting on this space are precisely those values in the set,
$$\left\{\tfrac{1}{6}\cdot\left\lfloor\tfrac{9i}{7}\right\rfloor: i \in \N\right\},$$
which are less than $k-1$ (each corresponding to a one-dimensional eigenspace).

The classical space, $S_k(\Gamma_0(49),\chi^{8-7k})$, is completely diagonalized by the $U_7$ operator over the field $K_2(\alpha)$. The slopes of $U_7$ acting on this space are precisely those values in the set,
$$\left\{\tfrac{1}{6}\cdot\left\lfloor\tfrac{9i+6}{7}\right\rfloor: i \in \N\right\},$$
which are less than $k-1$ (each corresponding to a one-dimensional eigenspace).

In both cases, each slope corresponds to a unique one-dimensional eigenspace.
\end{theorem}

\begin{proof}
First consider the case of $S_k(\Gamma_0(49),\chi^{7k-6})$ over $K_1$. In this case, Theorem \ref{Theorem:MainTheorem-Overconvergent} guarantees the existence of an overconvergent $U_7$ eigenform for each slope in the set 
$$\left\{\tfrac{1}{6}\cdot\left\lfloor\tfrac{9i}{7}\right\rfloor: i \in \N\right\}.$$
By \cite[Theorem 1.1]{coleman-overconvergent}, the eigenforms corresponding to those slopes which are strictly less than $k-1$ are actually classical. To count the number of such eigenforms, we let $f(i)=\lfloor{9i/7}\rfloor$ and attempt to solve $f(i)<6(k-1)$. It is easy to show by induction that for $r\geq0$ we have
$f(5+14r)=6(3r+1)$, $f(10+14r)=6(3r+2)$, and $f(14+14r)=6(3r+3)$. So every multiple of $6$ occurs in the increasing sequence, $(f(i))_{i\geq1}$, and the number of terms strictly less than $6(k-1)$ is given by
$$\begin{cases}
         9+14\left(\frac{k-3}{3}\right) ,&\text{if $k\equiv0\mod{3}$}\\
         13+14\left(\frac{k-4}{3}\right),&\text{if $k\equiv1\mod{3}$}\\
         4+14\left(\frac{k-2}{3}\right),&\text{if $k\equiv 2\mod{3}$.}
\end{cases}$$
Hence, this is the number of overconvergent $U_7$ eigenforms (up to scalar multiple) with slope strictly less than $k-1$, which by Coleman must be classical.

On the other hand, we can compute the dimension of $S_k(\Gamma_0(49),\chi^{7k-6})$ directly with Cohen-Oesterl\'e. Since $3^{28}\equiv 18\mod{49}$ and $3^{14}\equiv 30\mod{49}$,
\begin{align*}
\chi^{7k-6}(18)+\chi^{7k-6}(30)&=\zeta^{28(7k-6)}+\zeta^{14(7k-6)}\\
                                                       &=(-\zeta^7)^k+(\zeta^7-1)^k.
\end{align*}
But $-\zeta^7$ and $\zeta^7-1$ are just the two distinct primitive cube roots of unity. So the above expression evaluates to $2$ if $k\equiv 0\mod{3}$ and $-1$ otherwise. Taking into account the values of $\epsilon$, Cohen-Oesterl\'e then gives the following dimensions for $S_k(\Gamma_0(49),\chi^{7k-6})$.
$$\frac{14k-17}{3}+\begin{cases}
\tfrac{1}{3}(2),&\text{if $k\equiv 0\mod{3}$}\\
0(-1),&\text{if $k\equiv 1\mod{3}$}\\
-\tfrac{1}{3}(-1),&\text{if $k\equiv 2\mod{3}$}
\end{cases}$$
In each case, it is immediate that the dimension of the classical space is identical to the number of overconvergent eigenforms which have slope less than $k-1$ and hence are classical. Since we know from the classical theory that $S_k(\Gamma_0(49),\chi^{7k-6})$ does have a basis of cuspidal eigenforms for the full Hecke algebra, and since the eigenvalues are distinct, the theorem follows in this case.

The $K_2$ case is very similar. This time we let $f(i)=\lfloor{(9i+6)/7}\rfloor$ and find that $f(4+14r)=6(3r+1)$, $f(9+14r)=6(3r+2)$ and $f(14+14r)=6(3r+3)$. This results in the following formula for the number of slopes in the given set which are strictly less than $k-1$.
$$\begin{cases}
         8+14\left(\frac{k-3}{3}\right) ,&\text{if $k\equiv0\mod{3}$}\\
         13+14\left(\frac{k-4}{3}\right),&\text{if $k\equiv1\mod{3}$}\\
         3+14\left(\frac{k-2}{3}\right),&\text{if $k\equiv 2\mod{3}$.}
\end{cases}$$
Once again, this agrees with the dimension of the classical space by Cohen-Oesterl\'e, since
\begin{align*}
\chi^{8-7k}(18)+\chi^{8-7k}(30)&=\zeta^{28(8-7k)}+\zeta^{14(8-7k)}\\
                                                  &=(\zeta^7-1)^{1+k}+(-\zeta^7)^{1+k}.
\end{align*}
So the total dimension of the classical space is

$$\frac{14k-17}{3}+\begin{cases}
\tfrac{1}{3}(-1),&\text{if $k\equiv 0\mod{3}$}\\
0(-1),&\text{if $k\equiv 1\mod{3}$}\\
-\tfrac{1}{3}(2),&\text{if $k\equiv 2\mod{3}$.}
\end{cases}$$

\end{proof}

\section{Explicit Verification of the Main Theorem}\label{Section:Verify}

One way to quickly check that the theorem is at least reasonable is to compare the dimensions of various character subspaces of $S_k(\Gamma_1(49))$ with the numbers of slopes which are predicted by the theorem in those cases. William Stein has computed these dimensions in the first several cases, and the data is given on his website precisely as follows:

{\footnotesize
\begin{verbatim}
<49, 
[* 
<(0), [ 1 ], t^2 + 10*t^4 + 20*t^6 + 28*t^8 + 38*t^10 + 48*t^12 + 56*t^14 + 66*t^16>, 
<(1), [ 42 ], 8*t^3 + 18*t^5 + 27*t^7 + 36*t^9 + 46*t^11 + 55*t^13 + 64*t^15>, 
<(2), [ 21 ], 4*t^2 + 13*t^4 + 22*t^6 + 32*t^8 + 41*t^10 + 50*t^12 + 60*t^14 + 69*t^16>, 
<(3), [ 14 ], 9*t^3 + 17*t^5 + 27*t^7 + 37*t^9 + 45*t^11 + 55*t^13 + 65*t^15>, 
<(6), [ 7 ], 3*t^2 + 13*t^4 + 23*t^6 + 31*t^8 + 41*t^10 + 51*t^12 + 59*t^14 + 69*t^16>, 
<(7), [ 6 ], 5*t^3 + 15*t^5 + 24*t^7 + 33*t^9 + 43*t^11 + 52*t^13 + 61*t^15>, 
<(14), [ 3 ], t^2 + 10*t^4 + 19*t^6 + 29*t^8 + 38*t^10 + 47*t^12 + 57*t^14 + 66*t^16>, 
<(21), [ 2 ], 6*t^3 + 14*t^5 + 24*t^7 + 34*t^9 + 42*t^11 + 52*t^13 + 62*t^15> 
*]>,
\end{verbatim}}

\noindent
In each entry, the second number is the order of the group generated by $\psi(3)$ where $\psi$ is the character. Then the coefficient of $t^k$ represents the dimension of $S_k(\Gamma_0(49),\psi)$. We will compare this data with Theorem \ref{Theorem:MainTheorem-Classical} in the weight $2$ case, and invite the reader to ``spot check'' a few others.

When $k=2$, the theorem predicts that a basis of newforms for $S_2(\Gamma_0(49),\chi^{8})$ will be defined over $K_1$ and have slopes $\{1/6,2/6,3/6,5/6\}$. This agrees with the above data, because $<\zeta^8>$ has order $21$, and the coefficient of $t^2$ is $4$ in the corresponding polynomial. Similarly, we should have a basis of newforms for $S_2(\Gamma_0(49),\chi^{-6})$ defined over $K_2$ and with slopes $\{2/6,3/6,4/6\}$. Since $<\chi^{-6}>$ has order $7$ and the coefficient of $t^2$ is $3$ in the corresponding polynomial, this also matches.

This, however, does not confirm any of the {\em slopes}. Stein's dimensions are computed using Cohen-Oesterl\'e, and so this is essentially a check that we have incorporated Cohen-Oesterl\'e correctly. For an independent check of some actual slopes, we can compare with explicit values of $a_7$ which are known for the weight $2$ Hecke newforms (and again we take them from Stein's website). When $\psi(3)=\gamma$, a primitive $21^\text{st}$ root of unity, there is exactly one family of Galois conjugate weight $2$ newforms in $S_2(\Gamma_0(49),\psi)$. They are defined over the degree $4$ extension of $\Q(\gamma)$ generated by the following polynomial.
{\small
\begin{multline*}
x^4 + (\gamma^5 + 1)x^3 + (\gamma^{10} - 5\gamma^5 + 1)x^2\\
 + (\gamma^{11} - 4\gamma^{10} - \gamma^7 - \gamma^6 - 2\gamma^5 - \gamma^3 + 2\gamma^2 - \gamma)x\\
  + (2\gamma^{10} + \gamma^9 + \gamma^8 + \gamma^7 - \gamma^6 - \gamma^5 - \gamma^4 + \gamma^2 + \gamma + 1)
  \end{multline*}}
\noindent
Taking $a$ to be a root of the degree $4$ polynomial, the value of $a_7$ is then given explicitly by
{\small
\begin{multline*}
(\gamma^{11} - \gamma^{10} + \gamma^8 - \gamma^7 - \gamma^6 + \gamma^5 - \gamma^3 +\gamma^2 - 1)a^3\\
+ (\gamma^8 - \gamma^6 + \gamma^5 - \gamma^4 - \gamma^3 + \gamma^2)a^2\qquad\qquad\qquad\qquad\\
+ (4\gamma^{11} - \gamma^6 + \gamma^5 + 4\gamma^4 - \gamma^3 +\gamma^2 - \gamma)a\qquad\\
\qquad\qquad\qquad\qquad\qquad - (\gamma^{11} - \gamma^{10} - 3\gamma^9 + \gamma^8 - \gamma^7 - 2\gamma^6  + 2\gamma^5 + \gamma^4 - 3\gamma^3 + 2\gamma^2 + \gamma - 3).
\end{multline*}}
Our theorem applies in this case, since it gives the slopes (over $K_1$) of the weight $2$ newforms with character $\chi^8$, and $\zeta^8$ is a primitive $21^\text{st}$ root. If we let $\gamma=\zeta^8$ for consistency, we find the following roots of the degree $4$ polynomial over $K_1$.
\begin{align*}
a_1&=4+5\pi_1+1\pi_1^2+2\pi_1^3+3\pi_1^4+5\pi_1^5+6\pi_1^6+4\pi_1^7+4\pi_1^8+1\pi_1^9+1\pi_1^{10}+\cdots\\
a_2&=5+4\pi_1+2\pi_1^2+3\pi_1^3+4\pi_1^4+1\pi_1^5+5\pi_1^7+5\pi_1^8+3\pi_1^9+2\pi_1^{11}+\cdots\\
a_3&=4+1\pi_1+5\pi_1^2+4\pi_1^3+1\pi_1^4+6\pi_1^5+1\pi_1^6+3\pi_1^7+5\pi_1^8+6\pi_1^9+5\pi_1^{10}+\cdots\\
a_4&=5+5\pi_1^2+4\pi_1^3+4\pi_1^5+2\pi_1^6+2\pi_1^7+5\pi_1^8+6\pi1^{11}+\cdots
\end{align*}
Plugging these four values in for $a$ in the expression for $a_7$, we find $\pi_1$-adic valuations of $1$, $2$, $3$, and $5$. So the theorem is verified in this case.

Similarly, we can verify our weight $2$ slopes over $K_2$ by considering all forms in $S_2(\Gamma_0(49),\psi)$ where $\gamma=\psi(3)$ is a primitive $7^\text{th}$ root of unity. Since our theorem predicts the slopes of those eigenforms in $S_2(\Gamma_0(49),\chi^{-6})$, we must choose $\gamma=\zeta^{-6}$ for consistency. From Stein, we have three forms to consider. The first is defined over $\Q(\gamma)$ and has
$$a_7=2\gamma^5 + 2\gamma^4 + \gamma^3 + 2.$$
It is easy to check that $v_{\pi_2}(a_7)=3$ for this form. The other two are Galois conjugates defined over quadratic extension of $\Q(\gamma)$ generated by 
$$p(x)=x^2 - (\gamma^4 +\gamma)x - (\gamma^5 - \gamma^2 -\gamma).$$
Then, if $a$ is a root of $p(x)$, the value of $a_7$ is given explicitly by
$$a_7=(\gamma^3 - \gamma^2)a - (\gamma^4 - \gamma^2 - \gamma + 1).$$
Over $K_2$, we have the following two roots for $p(x)$.
\begin{align*}
a_1&=1+1\pi_2+6\pi_2^2+2\pi_2^3+1\pi_2^4+6\pi_2^5+6\pi_2^6+4\pi_2^7+3\pi_2^8+5\pi_2^9+5\pi_2^{10}+\cdots\\
a_2&=1+3\pi_2+6\pi_2^2+3\pi_2^3+1\pi_2^4+3\pi_2^5+1\pi_2^6+2\pi_2^7+1\pi_2^8+1\pi_2^9+1\pi_2^{10}+\cdots
\end{align*}
Setting $a=a_1$, we find that $v_{\pi_2}(a_7)=4$, and for $a=a_2$ we have $v_{\pi_2}(a_7)=2$. Thus, the theorem is verified in this case, since all three eigenforms are defined over $K_2$, and we have slopes of $\{2/6,3/6,4/6\}$.

\appendix

\section{Poles of $U_p(f)$ when $f$ is Meromorphic}\label{Section:Up-Meromorphic}

While this is not common in the literature, the operator $U_p$ can be applied to meromorphic forms for $X_1(M)$ via the geometric definition. 
As in \cite{katz}, we think of a weight $k$ modular form $f$ on $X_1(M)$ as a rule which assigns to each pair $(E,P)$, where $E$ is a generalized elliptic curve and $P$ is (roughly) a point of order $M$, a section of $\omega_E^{\otimes k}$. Then $U_p$ is defined by
$$(f|U_p)(E,P)=\frac{1}{p}\sum_{\phi}\phi^*(f(\phi E,\phi(P))),$$
where $\phi$ runs over all isogenies  $\phi:E\to\phi(E)$ of degree $p$ with $P\notin\ker(\phi)$ (and analogously for forms on $X_0(M)$).

We must apply $\tU_7$ to various meromorphic functions on $X_0(49)$ and eventually arrive at an explicit formula for $\tU_7(s^i)$. In order to justify our calculations, therefore, it is imperative that we be able to determine the orders of the poles of $f|U_p$, particularly when $f$ is supported on the cusps. The following lemma shows how we have done this using families of Tate curves. In order to simplify the exposition, we only prove the lemma here for (true) $U_p$ applied to functions on $X_0(p^2)$. The proof generalizes easily, however, to other weights and levels.

\begin{lemma}\label{Lemma:Up-Meromorphic}
Let $f$ be a function on $X_0(p^2)$.
\begin{enumerate}[(i)]
\item If $f$ has a pole of order $m$ at the cusp $\infty$, and no other poles, then 
$$(f|U_p)\geq-\lfloor\tfrac{m}{p}\rfloor(\infty)-\lfloor\tfrac{m}{p}\rfloor\sum(C_{p,i}).$$
\item If $f$ has a pole of order $m$ at the cusp $0$, and no other poles, then $f|U_p$ has a pole of order $pm$ at $0$ and no other poles.
\end{enumerate}
\end{lemma}

\begin{proof}
Fix a primitive $\zeta\in\mu_{p^2}$. Let $D_q$ denote the disk $|q|<1$.

To prove (i), suppose that $f$ is holomorphic everywhere except $\infty$, and that the canonical $q$-expansion of $f$ at $\infty$ is given by $\sum_n a_n q^n$ where $n\geq -m$. We may interpret the $q$-expansion as the value of $f$ on the family of Tate curves $f(K^*/\dia{q},\mu_{p^2})$. Using the geometric definition of $U_p$, we now compute $f|U_p$ on the family $(K^*/\dia{q^p},\mu_{p^2})$.
\begin{align*}
(f|U_p)(K^*/\dia{q^p},\mu_{p^2})&=\tfrac{1}{p}\sum\nolimits_{i}f(K^*/\dia{q^p,\zeta^{ip}q},\mu_{p^2})\qquad{i=0,\dots,p-1}\\
                                                &=\tfrac{1}{p}\sum\nolimits_{i}f(K^*/\dia{\zeta^{ip}q},\mu_{p^2})\\
                                                &=\tfrac{1}{p}\sum\nolimits_{i}\sum\nolimits_n a_n(\zeta^{ip}q)^n\\
                                                &=\sum\nolimits_n a_n\left(\tfrac{1}{p}\right)(1+\zeta^{np}+\cdots+\zeta^{(p-1)np})q^n=\sum\nolimits_n a_{np}q^{np}
\end{align*}
Thus we arrive at the familiar formula for the canonical $q$-expansion at infinity, $(f|U_p)(q)=\sum_n{a_{np}q^n}$, and in particular the order of the pole is at most $\lfloor\tfrac{m}{p}\rfloor$.

Next, we determine the order of the pole of $f|U_p$ at the cusp, $C_{p,i}$, by computing $f|U_p$ on the family of Tate curves $(K^*/\dia{q^p},\dia{\zeta q})$.
\begin{align*}
(f|U_p)(K^*/\dia{q^p},\dia{\zeta q})&=\tfrac{1}{p}\sum\nl_i f(K^*/\dia{q^p,\zeta^{pi}q},\dia{\zeta q}) \qquad{i=0,\dots,p-1}\\
                                                           &=\tfrac{1}{p}\sum\nl_i f(K^*/\dia{\zeta^{pi}q},\dia{\zeta q})\\
                                                           &=\tfrac{1}{p}\sum\nl_i f(K^*/\dia{\zeta^{pi}q},\mu_{p^2})\\
                                                           &=\tfrac{1}{p}\sum\nl_i\sum\nl_n a_n (\zeta^{pi}q)^n=\sum\nolimits_n a_{np}q^{np}
\end{align*}
Thinking of this series as a meromorphic function on $D_q$, the order of the pole at $q=0$ could be as much as $p\lfloor\tfrac{m}{p}\rfloor$. However, it is easy to see that the family of Tate curves in fact defined a degree $p$ map from $D_q$ into $X_0(p^2)$ taking $q=0$ to some $C_{p,i}$. Thus the pole of $f|U_p$ at $C_{p,i}$ has order at most $\lfloor\tfrac{m}{p}\rfloor$.

The proof for (ii) is similar. If $f$ has a pole of order $m$ at the cusp $0$, we know that $f(K^*/\dia{q^{p^2}},\dia{q})=a_{-m}q^{-m}+\cdots$ with $a_{-m}\neq0$ for $q\in D_q$ (this family defines a degree $1$ map from $D_q$ into $X_0(p^2)$ such that $q=0$ maps to the cusp $0$). The $p$ subgroups of $K^*/\dia{q^{p^2}}$ of order $p$ which are disjoint from $\dia{q}$ are $\mu_p$ and $\dia{\zeta^{pi}q^p}$ for $i=1,\dots,p-1$. Thus, applying the definition of $U_p$ we have
\begin{align*}
(f|U_p)(K^*/\dia{q^{p^2}},\dia{q})&=\tfrac{1}{p}\left(f(K^*/\dia{q^{p^2},\mu_p},\dia{q})+\sum\nl_i {f(K^*/\dia{q^{p^2},\zeta^{pi}q^p},\dia{q})}\right)\\
            &=\tfrac{1}{p}\left(f(K^*/\dia{q^{p^3}},\dia{q^p})+\sum\nl_i{f(K^*/\dia{\zeta^{pi}q^p},\dia{q})}\right)\\
            &=\tfrac{1}{p}\left(\sum_{n\geq-m}{a_nq^{pn}}+\sum\nl_i{f(K^*/\dia{\zeta^{pi}q^p},\dia{q})}\right).
\end{align*}
Each of the terms, $f(K^*/\dia{\zeta^{pi}q^p},\dia{q})$, must represent a holomorphic function near $q=0$, since this family of Tate curves is centered at one of the $C_{p,i}$ cusps. Thus, the $q$-expansion of $f|U_p$ at the family, $(K^*/\dia{q^{p^2}},\dia{q})$, begins with $a_{-m}q^{-mp}$.
\end{proof}




\end{document}